\documentclass[12pt]{amsart}
\usepackage{graphicx} 
\usepackage{soul} 
\usepackage{amscd,amssymb,mathtools}
\usepackage[arrow,matrix,graph,frame,poly,arc,tips]{xy}
\usepackage[dvipsnames]{xcolor}
\usepackage{graphicx}
\usepackage{calrsfs}
\usepackage[labelfont=rm]{subcaption}
\usepackage{tikz-cd} 
\usepackage{tikz}
\usetikzlibrary{decorations.markings}
\usetikzlibrary{shapes}
\usetikzlibrary{backgrounds}
\usetikzlibrary{calc}
\usetikzlibrary{shapes.misc, positioning}
\usepackage{mdwlist}
\usepackage{xfrac}
\DeclareCaptionSubType*{figure}

\usepackage{multirow}
\usepackage{enumerate}
\usepackage{verbatim}
\usepackage{comment}
\usepackage{todonotes}

  \usepackage{hyperref}

    \newcommand\ba{\begin{align*}}
    \newcommand\ea{\end{align*}}
    \newcommand\be{\begin{enumerate}}
    \newcommand\ee{\end{enumerate}}
    \newcommand\bpf{\begin{proof}}
    \newcommand\epf{\end{proof}}
    \newcommand\bpp{\begin{prop}}
    \newcommand\epp{\end{prop}}
    \newcommand\bpb{\begin{prob}}
    \newcommand\epb{\end{prob}}
    \newcommand\bd{\begin{defn}}
    \newcommand\ed{\end{defn}}
    \newcommand\bh{\begin{hint}}
    \newcommand\eh{\end{hint}}

    \newcommand\N{\mathbb{N}}

    \newcommand\R{\mathbb{R}}
    
    \newcommand\Q{\mathbb{Q}}
    \newcommand\bZ{\mathbb{Z}}
    \newcommand\Z{\mathbb{Z}}

    \newcommand\col{\mathrm{Col}}

    \DeclareMathOperator\Homeo{Homeo}

    \newcommand\mL{\mathcal{L}}

    \DeclareMathOperator\tp{tp}

    \DeclareMathOperator\ro{RO}
    \DeclareMathOperator\rc{RC}

    \DeclareMathOperator\Th{{Th}}

    \DeclareMathOperator\For{FOR}

\newcommand{\mgd}[2]{\{  {#1}\;|\; {#2} \} }
\newcommand{\Mgd}[2]{\big\{  {#1}\;\big|\; {#2} \big\} }
\newcommand{\MGD}[2]{\Big\{  {#1}\;\Big|\; {#2} \Big\} }
\newcommand{\BMGD}[2]{\Bigg\{  {#1}\;\Bigg|\; {#2} \Bigg\} }

\def\thetitle{Set theory, logic, and homeomorphism groups of manifolds}

     \theoremstyle{plain}
    
    \newtheorem{thm}{Theorem}[section]
    \newtheorem{lem}[thm]{Lemma}
    \newtheorem{lemma}[thm]{Lemma}
    \newtheorem{cor}[thm]{Corollary}
    \newtheorem{prop}[thm]{Proposition}

    \newtheorem*{claim*}{Claim}

    \theoremstyle{remark}

    \theoremstyle{definition}
    \newtheorem{defn}[thm]{Definition}
    \newtheorem{prob}{Problem}[section]

\begin{document}
\title\thetitle
    \date{\today}


        \author[J. Hanson]{James E. Hanson}
    \address{Department of Mathematics, Iowa State University, Ames,
    IA 50011}
    \email{jameseh@iastate.edu}
    \urladdr{https://james-hanson.github.io/}

    \author[T. Koberda]{Thomas Koberda}
    \address{Department of Mathematics, University of Virginia, Charlottesville, VA 22904-4137, USA}
    \email{thomas.koberda@gmail.com}
    \urladdr{https://sites.google.com/view/koberdat}
    
    \author[J. de la Nuez Gonz\'alez]{J. de la Nuez Gonz\'alez}
    \address{School of Mathematics, Korea Institute for Advanced Study (KIAS), Seoul, 02455, Korea}
    \email{jnuezgonzalez@gmail.com}

    \author[C. Rosendal]{Christian Rosendal}
    \address{Department of Mathematics, University of Maryland,
    College Park, MD 20742}
    \email{rosendal@umd.edu}
    \urladdr{https://sites.google.com/view/christian-rosendal}

\begin{abstract}
    We investigate the relationship between axiomatic set theory and the first-order theory of homeomorphism groups of manifolds in the language of group theory, concentrating on first-order rigidity and type versus conjugacy. We prove that under the axiom of constructibility (i.e.~{V=L}), homeomorphism groups of arbitrary connected manifolds are first-order rigid, and that the conjugacy class of a homeomorphism of a manifold is determined by its type. In contradistinction, under the regularity hypothesis that every projective set of reals has the Baire property, we show that in all dimensions greater than one there exist pairs of noncompact, connected manifolds whose homeomorphism groups are elementarily equivalent but which are not homeomorphic. We also show, under the same Baire-property hypothesis, that every manifold of positive dimension admits pairs of homeomorphisms with the same type which are not conjugate to each other. Projective determinacy implies the Baire-property hypothesis, so the corresponding consequences under PD follow immediately. Finally, we show that infinitary formulas do determine conjugacy classes of homeomorphisms and homeomorphism types of manifolds; specifically, the conjugacy class of a homeomorphism of an arbitrary manifold is determined by a single $L_{\omega_1\omega}$ formula. Similarly, the homeomorphism type of an arbitrary connected manifold is determined by a single $L_{\omega_1\omega}$ sentence.
\end{abstract}

 \maketitle
 \setcounter{tocdepth}{1}
\tableofcontents

\section{Introduction}

In this paper, we investigate the first-order theory of $\Homeo(M)$, the homeomorphism
group of an arbitrary
connected manifold $M$, and its relation to underlying axioms
in set theory. In~\cite{Rubin1989}, M.~Rubin conjectured that
under the set theoretic assumption of G\"odel constructibility,
if $M$ and $N$ are arbitrary, connected, boundaryless manifolds
with elementarily equivalent homeomorphism groups, then $M$ and $N$
are homeomorphic. In this paper, we
completely resolve Rubin's conjecture and furthermore show that set theoretic
axioms beyond ZFC are necessary.

Throughout this paper, $M$ will denote an arbitrary (not
necessarily compact) connected manifold of positive dimension and
with empty boundary; we will
assume that $M$ is second countable and Hausdorff. We write 
$\Homeo(M)$ for the group of all homeomorphisms of $M$, which will be viewed as a first-order structure in the  language $\mL$ of group theory. We adopt
the convention that this language has exactly two nonlogical symbols,
namely the binary multiplication operation and the identity. 

The \emph{theory} of $\Homeo(M)$, denoted $\Th(\Homeo(M))$, consists of all first-order $\mL$-sentences $\psi$ such that $\Homeo(M)\models\psi$. We will say
that $\Homeo(M)$ is \emph{first-order rigid} if for all manifolds $N$,
we have that $M$ is homeomorphic to $N$ if and only if $\Th(\Homeo(M))=\Th(
\Homeo(N))$. When the theories agree,
we also write $\Homeo(M)\equiv \Homeo(N)$, and say
that these two groups are \emph{elementarily equivalent}.
The first-order theory of homeomorphism groups of compact manifolds
was investigated in~\cite{KKdlN22,KKdlN2025,KdlN2023,KdlN2024}.

Similarly, if $g\in\Homeo(M)$, the \emph{type} of $g$, written $\tp(g)$,
consists of all first-order formulae $\psi(v)$  in one variable $v$
such that $\Homeo(M)\models\psi(g)$.

In this paper, we will compare two incompatible set-theoretic hypotheses. The first is G\"odel constructibility, denoted {V=L}; see~\cite{hinman-book,devlin-book} for background and discussion. Roughly, {V=L} is a restriction on the scope of the power set operation, and asserts that the only sets which exist at a particular level $V_{\alpha+1}$ of Zermelo's cumulative hierarchy are ones which are definable in the language of set theory, with parameters in the preceding level $V_\alpha$. Under {V=L}, the Axiom of Choice and the Continuum Hypothesis are both theorems. Importantly for us, under {V=L}, there is a projectively definable well-ordering of $\R$, or more generally of a definably presented Polish space.

Recall that a subset $A$ of a topological space $X$ has the \emph{Baire property} if it differs (by set-theoretic difference) from an open set by a \emph{meager set}, i.e.~ one contained in a countable union of nowhere dense closed sets.
The second hypothesis used in this paper is the following projective regularity assertion:
\[
\mathsf{PBP}:\quad\text{every projective set of reals has the Baire property.}
\]
Equivalently for the applications below, every projective subset of the standard projective code spaces used in this paper has the Baire property. The better known regularity axiom \emph{projective determinacy} (PD) implies $\mathsf{PBP}$; however, the proofs of non-classification below use only $\mathsf{PBP}$ and the associated category zero-one laws, not determinacy itself.

Our first results relate first-order rigidity of homeomorphism groups
of noncompact manifolds to extensions of ZFC:

\begin{thm}\label{thm:vl-rigid}
    Suppose {V=L} holds.
    Then homeomorphism groups of manifolds are first-order
    rigid; that is, for all pairs $M,N$ of connected manifolds without
    boundary, we have $\Homeo(M)\equiv \Homeo(N)$ if and only if $M$ and
    $N$ are homeomorphic.
\end{thm}

By contrast:

\begin{thm}\label{thm:pd-notrigid}
    Suppose $\mathsf{PBP}$ holds. Then for all $n\geqslant 2$, there exist
    manifolds $M$ and $N$ of dimension $n$ that are not
    homeomorphic, but such that
    $\Homeo(M)\equiv \Homeo(N)$.
\end{thm}
Since PD implies $\mathsf{PBP}$, the corresponding statement under projective determinacy follows immediately.

In~\cite{KKdlN22}, the second and third authors
together with Kim proved that
first-order rigidity holds for compact manifolds within ZFC, and so the
examples furnished by Theorem~\ref{thm:pd-notrigid} are necessarily
noncompact.

Theorem~\ref{thm:pd-notrigid} does not require producing examples in
all dimensions; we will show that it is a theorem of ZFC that
failure of first-order rigidity in dimension $d$ generally
causes failure
of first-order rigidity in all dimensions above $d$:

\begin{thm}~\label{thm:pd-dim-implication}
Suppose $M$ and $N$ are manifolds of some fixed dimension $d$ and let $S^k$ denote the $k$--dimensional sphere, $k\geqslant 1$. Then 
$$
\Homeo(M)\equiv\Homeo(N) \quad \Rightarrow\quad \Homeo(M\times S^k)\equiv \Homeo(N\times S^k).
$$
In particular, if there are $M$ and $N$ such that $\Homeo(M)\equiv\Homeo(N)$ and $M\times S^k\ncong N\times S^k$, then first-order rigidity fails in dimension $d+k$.
\end{thm}

In Theorem~\ref{thm:pd-dim-implication}, the assumption that
both $M$ and $N$ are not homeomorphic and
$M\times S^k$ and $N\times S^k$ are also not homeomorphic is not due to
an idle worry; two non-homeomorphic manifolds can become homeomorphic
after taking a Cartesian product with a fixed manifold. However,
the choice of $S^k$ in the statement of 
Theorem~\ref{thm:pd-dim-implication} is not crucial; it suffices
to consider any manifold $X$ that is definably presented as a Polish
space (see Section~\ref{ss:effective}), such that taking a Cartesian
product with $X$ is injective on homeomorphism classes.

As is well-known and also follows from the results above, the assumptions {V=L} and $\mathsf{PBP}$ are incompatible. However, {V=L} has the same consistency strength as ZF, meaning that if ZF is consistent then so is ZF + {V=L}. On the other hand, $\mathsf{PBP}$ is a regularity hypothesis rather than a determinacy axiom. Shelah proved that every model of ZFC has a forcing extension satisfying $\mathsf{PBP}$~\cite{Shelah1984}; hence the failure of first-order rigidity has no consistency strength beyond ZFC. We record the resulting consequence explicitly:

\begin{thm}\label{thm:consistent-intro}
    Every model of ZFC has a forcing extension in which, for each
    $n\geqslant 2$, there is a pair of non-homeomorphic $n$--manifolds
    with elementarily equivalent homeomorphism groups.
\end{thm}

\begin{cor}
The first-order rigidity of homeomorphism groups of manifolds is independent of ZFC.
\end{cor}

Whereas Shelah’s theorem gives the full every-model statement, our proof of the corresponding relative consistency statement is for countable transitive models.

The above results concern the question of whether one can distinguish different manifolds by the theory of their homeomorphism groups. Similarly, one may ask if one can distinguish different homeomorphisms of the same manifold via their type. We will call a homeomorphism $g\in\Homeo(M)$ \emph{type rigid} if for
all $h\in\Homeo(M)$, we have $\tp(g)=\tp(h)$ if and only if $g$ and $h$
are conjugate in $\Homeo(M)$.

For specific manifolds, one can often
find many homeomorphisms which are type rigid, by straightforward
application of the machinery in~\cite{KKdlN2025}; for instance,
a homeomorphism $h$ of a sphere $S^1$ with north--south dynamics
is type rigid.
In fact, the conjugacy class of $h$ is \emph{isolated} by a
single formula.
That is, there is a formula $\phi(v)$ such that $\phi(h)$ holds,
and any other homeomorphism of $S^1$ satisfying $\phi$ is conjugate to
$h$ in $\Homeo(S^1)$. It follows then that the type
of $h$ is isolated by the single formula $\phi$, i.e.~for all other
formulae $\psi(v)\in\tp(h)$, we have
$\Homeo(S^1)\models \forall v\,\big(\phi(v)\to\psi(v)\big)$. 

From a model-theoretic point of view,
type rigidity measures the homogeneity of the structure $\Homeo(M)$;
indeed, two conjugate elements of $\Homeo(M)$ will certainly have
the same type, and structures in which the type of an element
determines the automorphism orbit of that element are
\emph{$1$--homogeneous}. For compact manifolds at least, a result of 
Whittaker~\cite{whittaker} shows that the automorphism group of
$\Homeo(M)$ coincides with the inner automorphisms of $\Homeo(M)$,
i.e.~conjugation. We show that for all manifolds, $1$--homogeneity
of $\Homeo(M)$ depends on the set theory being used:

\begin{thm}\label{thm:vl-type-rigid}
    Suppose {V=L} holds, and let $M$ be an arbitrary connected manifold.
    Then every $g\in\Homeo(M)$ is type rigid.
\end{thm}

By contrast:

\begin{thm}\label{thm:pd-not-type-rigid}
    Suppose $\mathsf{PBP}$ holds. Then for all connected manifolds $M$, there exist pairs
    $g,h\in\Homeo(M)$ such that $\tp(g)=\tp(h)$ but such that $g$ and $h$
    are not conjugate in $\Homeo(M)$.
\end{thm}
Again, the corresponding statement under PD is an immediate corollary.

Finally, we consider rigidity and type rigidity within ZFC with infinitary
logics. Specifically, we will be interested in $L_{\omega_1\omega}$ logic.
Formulae in $L_{\omega_1\omega}$ have the same signature as in classical
first-order logic, and terms and atomic formulae are defined identically.
Negations of formulae are
again formulae, and whenever $\phi(x)$ is a formula with a free
variable $x$
then $\exists x\;\phi(x)$ is also a formula. The difference with
classical logic lies in the fact that countable disjunctions and countable
conjunctions of formulae are again formulae. Formulae that are
$L_{\omega_1\omega}$ play an important role in investigating the
descriptive set theory of spaces of countable models of theories; see~\cite{marker-infinitary} for a detailed discussion.

\begin{thm}\label{thm:infinitary-homeo}
    Let $M$ be an arbitrary connected manifold.
    There is a $L_{\omega_1\omega}$ sentence $\psi_M$ such that for all
    manifolds $N$, we have $\Homeo(N)\models\psi_M$ if and only if $M$ and
    $N$ are homeomorphic.
\end{thm}

Moreover:

\begin{thm}\label{thm:infinitary-type}
    Let $M$ be an arbitrary connected manifold and let $g\in\Homeo(M)$.
    Then there is an $L_{\omega_1\omega}$ formula $\psi_g(x)$ such that
    for all $h\in\Homeo(M)$, we have $\Homeo(M)\models\psi_g(h)$ if and only
    if $h$ is conjugate to $g$ in $\Homeo(M)$.
\end{thm}


\section{Definability in second-order arithmetic}

We recall a few notions from descriptive set theory,
adapted to the various spaces that we will need; see~\cite{moschovakis-book,kechris-book,gao-book}, for instance.
Recall that a {\em Polish space} is a separable
topological space $X$ whose topology can be induced by a complete
metric on $X$. Thus, generally, the metric is not part of the given data.
In this section, we discuss presentations of Polish spaces
and relate them to parameter-free definability in second-order arithmetic.


\subsection{Definable presentations of Polish spaces and definability}\label{ss:effective}
In this section, we will gather some background on definably
presented Polish spaces; the fundamental reason for this discussion
is that in the sequel,
we wish for all of our constructions to be carried out
\emph{parameter-free in second-order arithmetic}. That is, we will
be given the two-sorted structure
$(\N,2^{\N},+,\times,<,\in)$, i.e.~second-order
arithmetic, and we will require all implicit formulae relating to
Polish spaces to be parameter-free. The reader may find a
discussion of the content of this subsection
in~\cite{moschovakis-book,simpson-book}.

\subsubsection{Integers and rational numbers}
It is a standard fact that the structure $(\Z,+,\times,<)$ can be
parameter-free interpreted in standard first-order arithmetic 
$(\N,+,\times,<)$ as a definable set of pairs of natural numbers up
to a suitable parameter-free
definable equivalence relation. A similar interpretation
of the structure $(\Q,+,\times,<)$ can be carried out. Thus, via
an effective pairing \[\langle\cdot,\cdot\rangle\colon
\N\times\N\longrightarrow \N\]
both $\Q$ and $\Z$ can be viewed as parameter-free definable subsets of
$\N$ up to equivalence, and the operations/relations of addition,
multiplication, and order are arithmetically definable. The Euclidean
metric $d$ on $\Q$ is a parameter-free definable function.

\subsubsection{Sequences and reals}
The set of codes for sequences of natural numbers is parameter-free definable in
second-order arithmetic. Indeed, a sequence $\sigma$ of natural
numbers can be coded as a subset of $\N$ via
\[\sigma\mapsto X_{\sigma}=\{\langle n,\sigma(n)\rangle\mid n\in\N\}.\]
By the definability of the predicates $(+,\times,<)$ over $\Q$,
the set of Cauchy sequences with respect to the metric $d$
in $\Q$ is parameter-free definable. From this,
we obtain a parameter-free definable sort $\R$ consisting of convergent
Cauchy sequences, together with a definable extension of the metric $d$
and the predicates $(+,\times,<)$ to $\R$. In the sequel, $\R$ will
serve as a prototypical definably presented Polish space.

\subsubsection{Definably presented Polish spaces}
An \emph{definably presented Polish space} $X$ is given by a countable
set $D=\{x_n\}_{n\in\N}$ and a function
\[d=d_X\colon D\times D\longrightarrow \R_{\geqslant 0}\]
satisfying the axioms for a metric, such that the map
$(m,n)\mapsto d(x_m,x_n)$ is parameter-free definable in
second-order arithmetic. The points of $X$ are identified with the
(parameter-free definable) set of Cauchy sequences in $D$, up
to the natural equivalence relation.

Observe that there  are natural maps
\[n\mapsto x_n,\quad A\subseteq\N\mapsto X_A=\{x_n\mid n\in A\}.\]

The parameter-free definable subsets of the space $(X,d)$ are, by
definition, given by the parameter-free
definable sets in the structure $(\N,2^{\N},+,\times,<,\in)$, via the 
maps above.

\subsubsection{Coding open and closed sets}\label{sss:closed-open}

Let $X$ be a definably presented Polish space, with a dense
subset $\{x_n\}_{n\in\N}$ and definable metric $d$. An open subset $U$
of $X$ is given as a union of countably many metric balls of rational
radius contained in $U$, each of which is centered at a point in
$\{x_n\}_{n\in\N}$. A sequence of such balls is simply given by a
subset $A\subseteq \N$ and sequence $\sigma$
of pairs of the form
$(n,q)$ encoding an open $d$--ball of radius $q\in \Q_{>0}$ about
$x_n$, with $n\in A$. Thus, $x\in U$ if and only if there exists an
$(n,q)\in\sigma$ such that $d(x,x_n)<q$. The set of codes for subsets of $\N$ and the set of codes for sequences
of positive rationals are parameter-free definable, so codes for open subsets
of $X$ are parameter-free interpretable. Note that we are not actually
defining open subsets of $X$, but rather codes for them. It is
a straightforward exercise to write down a definable equivalence relation
for when two codes represent the same open set, and
a predicate expressing when a point
belongs to an open set.

Since closed sets are complements
of open sets, the closed subsets of $X$ are also
parameter-free interpretable. Compact subsets can be characterized by
sequential compactness; we omit the details.
A single open ball of rational radius $q$ about a point $x_n$ can be coded
by a constant sequence $(n,q)$. This, together with an
effective pairing $\N\times\Q_{>0}\longrightarrow \N$ gives a
parameter-free recursively enumerated sequence of codes for a
basis for the topology on $X$; we will
reserve $\{\mathcal U_n\}_{n\in\N}$ for such a basis. The predicate $x\in\mathcal U_i$
is parameter-free definable whenever $X$ is a definably presented Polish space.

\subsubsection{Some commonly used Polish spaces}

We will need some explicit definably presented Polish spaces. In
each case, it is clear that the Polish space is in fact definably
presentable.

\begin{itemize}
    \item For each $d\in\N$, the Euclidean space $\R^d$.
    \item For each $\R^d$ and each parameter-free definable 
    $r\in\R_{\geqslant 0}$,
    the closed ball $B^d(r)$ of radius $r$ centered at the origin; here, we
    implicitly use the standard Euclidean metric on $\R^d$.
    \item The intersection of $B^d(1)$ with coordinate planes in $\R^d$ (which
    are then homeomorphic to balls of smaller dimension).
    \item For each $d$, the $d-1$--dimensional sphere $S^{d-1}=\partial \big(B^d(1)\big)$.
    \item For all $d,d'\geqslant 1$, the Banach space $C(B^d(1),\R^{d'})$ of continuous functions $B^d(1)\to \R^{d'}$ with the usual supremum metric and the space
    $C\big(B^d(1),\R^{d'})^\Z$ of bi-infinite sequences of such functions with an appropriate product metric. 
    To lighten the notation, when the dimensions are obvious from the context,  we shall often simply write $B$, $E$, $C(B,E)$ and $C(B,E)^\Z$ for $B^d(1)$, $\R^{d'}$, $C\big(B^d(1),\R^{d'})$, respectively $C\big(B^d(1),\R^{d'})^\Z$. 
\end{itemize}

\subsection{Parameter-free projective definability}\label{ss:proj-defin}
If $X$ is a Polish space, we let $K(X)$ denote the space of all compact subsets of $X$ equipped with the {\em Vietoris topology}, which is the topology generated by sets of the form
$$
\Mgd{K\in K(X)}{K\cap U\neq \emptyset},\qquad \Mgd{K\in K(X)}{K\subseteq U}, 
$$
where $U$ ranges over open subsets of $X$. Similarly, $F(X)$ is the {\em Effros-Borel space} consisting of all closed subsets of $X$ equipped with the $\sigma$-algebra generated by sets of the form 
$$
\Mgd{F\in F(X)}{F\cap U\neq \emptyset} 
$$
with $U$ varying over open subsets of $X$. The space $K(X)$ is Polish, whereas $F(X)$ is standard Borel. Moreover, the canonical inclusion of $K(X)$ into $F(X)$ is a Borel embedding and $K(X)$ is a Borel subset of $F(X)$.

Suppose that $X$ is a definably presented Polish space. By the
discussion in Subsection~\ref{sss:closed-open}, the collections
$F(X)$ and $K(X)$ are parameter-free interpretable (i.e.~codeable in a way that
depends only on the presentation of $X$), as are the open
sets in $X$. From a fixed open set $U\subseteq X$, the collections
of compact or closed sets meeting or contained in $U$ are definable
in the interpreted sorts $K(X)$ and $F(X)$ respectively,
using only $U$ as a parameter.

If $X$ is a definably presented Polish space, then the \emph{projective}
sets are the subsets of $X$ which are definable with parameters. A subset which
is definable without parameters, we will call an \emph{parameter-free projective set}.
It is straightforward that parameter-free projective sets are closed under complements, finite unions,
and finite intersections. It is straightforward to show that these functions and sets, viewed
as relations, are parameter-free projective whenever $X$ is definably
presented.
\begin{enumerate}
    \item $\Mgd{(x,F)\in X\times F(X)}{x\in F}$,
    \item $\Mgd{(F_1,F_2)\in F(X)\times F(X)}{F_1\subseteq F_2}$
    \item $f\in C(B,E)\mapsto f[B]\in K(E)$,
    \item $(F_i)_{i\in \Z}\in F(X)^\Z\mapsto \overline{\bigcup_{i\in \Z}F_i}\in F(X)$.
\end{enumerate}

We have the following:

\begin{lemma}\label{lem:projective-property}
    Suppose $X$ is a definably presented Polish space.
    The following properties and relations are
    parameter-free projective.
    \begin{enumerate}
    \item $(x_n)_{n=1}^\infty\in X^\N$ is a sequence satisfying $x_1=\lim_{n\to \infty}x_n$,
        \item $F\in F(X)$ is the code for a connected set,
        \item $f\in C(B,E)$ is an injective function,
        \item $G\in F(X\times X)$ is the code for the
        graph of a homeomorphism between the two sets $F_1, F_2\in F(X)$. 
        \item $G\in F(X\times X)$ is the code for the
        graph
        of a homeomorphism of pairs $(F_1,A_1)\longrightarrow
        (F_1,A_2)$, for $A_i\subseteq F_i$ and $A_1,A_2,F_1,F_2\in
        F(X)$.
        \item $G\in F(X\times X)$ is the code for the graph
        of a homeomorphism of triples \[(F_1,A_1,x_1)\longrightarrow
        (F_1,A_2,x_2),\] for $x_i\in A_i\subseteq F_i$
        and $A_1,A_2,F_1,F_2\in
        F(X)$.
    \end{enumerate}
\end{lemma}

\begin{proof}
   Note that $F\in F(X)$ is connected if and only if
   $$
   \forall F_1\forall F_2\in F(X)\;\big( F_1\cup F_2\neq X \text{ or } F\cap F_1\cap F_2\neq  \emptyset\text{ or }F\cap F_1=\emptyset \text{ or }F\cap F_2=\emptyset\big), 
   $$
   which is clearly parameter-free definable in second-order arithmetic.

   Similarly, $f\in C(B,E)$ is injective if and only if 
   $$
   \forall x\in X\; \forall y\in X\; \big(x=y \text{ or } f(x)\neq f(y)\big).
   $$

   We have that $G\in F(X\times X)$ is the graph of a homeomorphism between the two sets $F_1, F_2\in F(X)$ exactly when
   \[\begin{split}
       &G \text{ is the graph of a bijection between $F_1$ and $F_2$ and}\\
      & \forall \big((x_n,y_n)\big)_{n\in \N}\; \Big(\forall n\; (x_n,y_n)\in G \to \big(x_1=\lim_{n\to \infty}x_n\leftrightarrow y_1=\lim_{n\to \infty}y_n\big)\Big)
   \end{split}\]
    which is clearly
   parameter-free definable in second-order arithmetic.
   The last two properties are established similarly. 
\end{proof}

To address the definability of the space of manifolds, we will need a good version of the Whitney Embedding Theorem for (possibly noncompact) topological manifolds that furnishes good local behavior, namely \emph{local topological flatness}: a $d$--dimensional submanifold is locally flat, in the sense that it is locally the intersection of a $d$--dimensional coordinate plane with the ambient space. For smooth submanifolds this is barely an assumption, and it is also true that a $d$--dimensional topological manifold is locally flatly embedded in some finite dimensional Euclidean space, the dimension of which depends only on $d$. It seems such a statement is difficult to locate explicitly in the literature, and so we provide a proof.

\begin{lemma}
\label{lem:locally-flat-euclidean-embedding}
Let \(M\) be a connected Hausdorff second-countable topological
\(n\)-manifold without boundary. Then \(M\) admits a topologically locally
flat embedding
\[
        F:M\hookrightarrow \mathbb R^{N(n)},
        \qquad
        N(n):=2n+1+n(n+1)=n^2+3n+1 .
\]
In particular, the ambient Euclidean dimension can be chosen to depend
only on \(n=\dim M\).
\end{lemma}

\begin{proof}
We write \(\dim_{\mathrm{cov}}\) for covering dimension. The hypotheses imply
that \(M\) is separable and metrizable, hence paracompact and normal.

First note that
\[
        \dim_{\mathrm{cov}} M\leq n .
\]
Indeed, one proves by induction on \(n\) that the small inductive dimension
satisfies \(\operatorname{ind}M\leq n\). We refer the reader to standard texts on dimension theory such as ~\cite{EngelkingDimension} for a discussion of inductive dimension.

Given \(x\in M\) and an open
neighborhood \(O\) of \(x\), choose a coordinate chart
\(\phi:C\to \phi(C)\subset \mathbb R^n\) with \(x\in C\), and choose a
Euclidean ball \(B\subset \phi(C\cap O)\) centered at \(\phi(x)\) whose
closed ball is still contained in \(\phi(C\cap O)\). Then
\[
        U:=\phi^{-1}(B)
\]
is an open neighborhood of \(x\) with \(U\subset O\), and its frontier in
\(M\) is homeomorphic to \(S^{n-1}\), with the convention that
\(S^{-1}=\varnothing\). By induction this frontier has small inductive
dimension at most \(n-1\). Thus \(\operatorname{ind}M\leq n\). Since \(M\)
is separable metrizable, Engelking's coincidence theorem gives
\[
        \operatorname{ind}M=\dim_{\mathrm{cov}}M ,
\]
so \(\dim_{\mathrm{cov}}M\leq n\); see
\cite[Theorem 1.7.7]{EngelkingDimension}.

Let \(\mathcal C\) be the open cover of \(M\) by coordinate domains.
By Ostrand's theorem, applied to the metrizable space \(M\) with
\(\dim_{\mathrm{cov}}M\leq n\), there are \(n+1\) discrete families of open
sets
\[
        \mathcal V_0,\ldots,\mathcal V_n
\]
such that
\[
        \mathcal V:=\bigcup_{a=0}^n \mathcal V_a
\]
covers \(M\) and refines \(\mathcal C\); see
\cite[Theorem 1]{Ostrand1965}. Here \emph{discrete} means that each point of
\(M\) has a neighborhood meeting at most one member of the family. In
particular, the finite union \(\mathcal V\) is locally finite.

For each \(V\in\mathcal V_a\), choose a coordinate chart
\[
        \phi_V:C_V\longrightarrow \phi_V(C_V)\subset \mathbb R^n
\]
such that \(V\subset C_V\). Since \(M\) is paracompact Hausdorff, choose a
partition of unity
\[
        \{\rho_V\}_{V\in\mathcal V}
\]
subordinate to the locally finite open cover \(\mathcal V\), with
\[
        \operatorname{supp}(\rho_V)\subset V
\]
for each \(V\); see, for example, \cite[Theorem 41.7]{munkres}.

Put
\[
        c:=\frac{1}{n+2}.
\]
Because each point of \(M\) belongs to at most one member of each discrete
family \(\mathcal V_a\), at most \(n+1\) of the functions \(\rho_V\) are
nonzero at any given point. Since they sum to \(1\), the open sets
\[
        U_V:=\{x\in M:\rho_V(x)>c\}
\]
cover \(M\).

Choose a continuous function \(\theta:[0,1]\to[0,1]\) such that
\[
        \theta(t)=0 \quad \text{for } t\leq c/2,
        \qquad
        \theta(t)=1 \quad \text{for } t\geq c.
\]
Define
\[
        \lambda_V:=\theta\circ\rho_V .
\]
Then \(\operatorname{supp}(\lambda_V)\subset V\subset C_V\), and
\(\lambda_V=1\) on \(U_V\).

Define a continuous map \(h_V:M\to\mathbb R^n\) by
\[
        h_V(x)=
        \begin{cases}
        \lambda_V(x)\phi_V(x), & x\in C_V,\\
        0, & x\notin C_V .
        \end{cases}
\]
This is continuous because \(\operatorname{supp}(\lambda_V)\subset C_V\).
For each color \(a=0,\ldots,n\), define
\[
        f_a:M\longrightarrow \mathbb R^n,
        \qquad
        f_a(x):=\sum_{V\in\mathcal V_a} h_V(x).
\]
The sum is locally finite; in fact, near any point, at most one summand
from the discrete family \(\mathcal V_a\) is nonzero. Moreover, if
\(V\in\mathcal V_a\), then
\[
        f_a|_{U_V}=\phi_V|_{U_V}.
\]
Thus at every point of \(M\), at least one of the finitely many maps
\(f_0,\ldots,f_n\) is literally a coordinate chart on a neighborhood of
that point.

Since \(M\) is separable metrizable and \(\dim_{\mathrm{cov}}M\leq n\),
the Menger--Nöbeling embedding theorem gives a topological embedding
\[
        e:M\hookrightarrow \mathbb R^{2n+1};
\]
see \cite[Theorem 1.11.4]{EngelkingDimension}. Now define
\[
        F:M\longrightarrow
        \mathbb R^{2n+1}\times(\mathbb R^n)^{n+1}
        \cong \mathbb R^{n^2+3n+1}
\]
by
\[
        F(x):=\bigl(e(x),f_0(x),\ldots,f_n(x)\bigr).
\]
This is a topological embedding, since projection onto the first factor
recovers the embedding \(e\).

It remains to check local flatness. Fix \(x\in M\). Since the sets \(U_V\)
cover \(M\), choose \(V\in\mathcal V_a\) with \(x\in U_V\). On \(U_V\),
the map \(f_a\) equals the coordinate chart \(\phi_V\). After permuting
the Euclidean coordinates of the target, write
\[
        F(y)=\bigl(f_a(y),H(y)\bigr)
        \in \mathbb R^n\times \mathbb R^{N(n)-n}
\]
for \(y\in U_V\), where \(H\) denotes all the remaining coordinate
functions. Let
\[
        A:=\phi_V(U_V)\subset\mathbb R^n
\]
and define
\[
        g:A\longrightarrow \mathbb R^{N(n)-n},
        \qquad
        g(u):=H(\phi_V^{-1}(u)).
\]
Then
\[
        F(U_V)=\{(u,g(u)):u\in A\},
\]
the graph of the continuous map \(g\).

Because \(F\) is an embedding and \(U_V\) is open in \(M\), there is an open
neighborhood \(\Omega\subset\mathbb R^{N(n)}\) of \(F(x)\) such that
\[
        \Omega\cap F(M)=F(U_V).
\]
Choose product neighborhoods \(A_0\subset A\) of \(\phi_V(x)\) and
\(W\subset\mathbb R^{N(n)-n}\) of \(g(\phi_V(x))\) such that
\[
        A_0\times W\subset \Omega
        \quad\text{and}\quad
        g(A_0)\subset W .
\]
On \(A_0\times W\), the homeomorphism
\[
        (u,v)\longmapsto (u,v-g(u))
\]
carries
\[
        (A_0\times W)\cap F(M)
        =
        \{(u,g(u)):u\in A_0\}
\]
onto
\[
        A_0\times\{0\}.
\]
Thus near \(F(x)\), the pair
\[
        \bigl(\mathbb R^{N(n)},F(M)\bigr)
\]
is homeomorphic to
\[
        \bigl(\mathbb R^{N(n)},\mathbb R^n\times\{0\}\bigr)
\]
locally. Hence \(F\) is topologically locally flat.
\end{proof}

We can now construct a parameter-free projective space of $d$--dimensional submanifolds of $\R^{d'}$, which contains homeomorphic copies of all $d$--dimensional connected manifolds whenever $d'$ is at least $N(d)$, as in Lemma~\ref{lem:locally-flat-euclidean-embedding}.

\begin{lem}\label{lem:manifold-projective}
Fix dimensions $1\leqslant d\leqslant d'$.  The collection 
$$
\mathrm{manif}_d=\Mgd{   F\in F(\R^{d'})   }{F \text{ is a connected $d$-dimensional submanifold}}
$$
is parameter-free projective.
\end{lem}
\begin{proof}
    Since connectedness is already parameter-free projective, it suffices to show
    that being a $d$--dimensional submanifold of $\R^{d'}$ is parameter-free projective.
    Let $\R^d$ be embedded in $\R^{d'}$  on the first $d$ coordinates, whereby $B^d(1)=\R^d\cap B^{d'}(1)$. A connected set $F_1\in F(\R^{d'})$ is a locally flat $d$--dimensional
    submanifold if and only if, for all $x\in F_1$, there exists a
    $F_2\in F(\R^{d'})$ such that we obtain a homeomorphism of triples
    \[(F_2,F_1\cap F_2,x)\cong (B^{d'}(1),B^d(1),0).\]
    The conclusion now follows by Lemma~\ref{lem:projective-property}.
\end{proof}

If one wishes to avoid the issue of local flatness of topological submanifolds of Euclidean spaces, one can modify the proof we have given to express that $F$ is locally the graph of an embedding of $\R^d$ into $\R^{d'}$. The details are not difficult to make explicit.

The following corollaries are straightforward.

\begin{cor}\label{cor:manifold-homeo}
    The homeomorphism relation on pairs of $d$--dimensional submanifolds
    of $\R^{d'}$ is parameter-free projective.
\end{cor}

\begin{cor}\label{cor:homeogrp-projective}
    Let $E=\R^{d'}$ be a fixed Euclidean space. The subset
    \[\mathrm{HGrp}_d\subseteq F(E)\times F(E\times E)\]
    consisting of pairs $(M,G)$, for which $M\in \mathrm{manif}_d$ and
    $G$ is the graph of a homeomorphism of $M$, is parameter-free projective.
\end{cor}

\begin{cor}\label{cor:seq-manifold}
    The set 
    $$
    \mathfrak m_d=\BMGD{(f_n)\in C(B,E)^\Z}{\;
    \overline{\bigcup_{n\in\Z} f_n[B]}\;\in\;  \mathrm{manif}_d }
    $$ 
    is
    parameter-free projective. 
\end{cor}

Recall that we fixed a recursively enumerated 
set $\{\mathcal U_i\}_{i\in\N}$  of codes for a countable
basis for the topology on
$C(B,E)^\Z$.
For an element $(f_n)\in\mathfrak m_d$, we may think of $(f_n)$ as
encoding an atlas for an embedded submanifold of $E$, and the indices $i$ for which
$(f_n)$ belongs to $\mathcal U_i$ can be viewed as an address for $(f_n)$.

\begin{prop}\label{prop:conj-projective}
    Let $\mathrm{HGrp}_d$ be as in
    Corollary~\ref{cor:homeogrp-projective}. The relation
    \[\mathrm{conj}_d\subseteq \mathrm{HGrp}_d\times \mathrm{HGrp}_d\]
    consisting of pairs
    $(M_1,G_1)$ and $(M_2,G_2)$, for which there exists a homeomorphism
    $\phi\colon M_1\rightarrow M_2$ conjugating the homeomorphism
    with graph $G_1$ to the homeomorphism with graph $G_2$, is
    parameter-free projective.
\end{prop}
\begin{proof}
    Write $(x_1,y_1)\in G_1$ and $(x_2,y_2)\in G_2$ for typical points.
    The existence of the homeomorphism $\phi$ is equivalent to the
    existence of a graph $G_3$ of a homeomorphism between $M_1$ and
    $M_2$ such that 
    \[
    \forall(x_i,y_i)\;\Big(\big((x_1,y_1)\in G_1\wedge (y_1,y_2)\in G_3\big)
    \leftrightarrow \big((x_1,x_2)\in G_3\wedge(x_2,y_2)\in G_2\big)\Big).
    \]
    It is clear then that $\mathrm{conj}_d$ is parameter-free projective.
\end{proof}

To investigate type rigidity of homeomorphisms of manifolds, we will
require a parameter-free projective predicate encoding pairs of atlases for
a manifold $M$, where one atlas is twisted by a fixed homeomorphism
of $M$. We write $\mathrm{mark}_d$ for the set of pairs
$((f_n),(g_n))\in\mathfrak m_d\times \mathfrak m_d$ such that:
\begin{enumerate}
    \item \[M:=\overline{\bigcup_{n\in\Z} f_n[B]}=
    \overline{\bigcup_{n\in\Z} g_n[B]};\]
    \item There exists a $G$ such that $(M,G)\in\mathrm{HGrp}_d$
    and such that for all $x\in B$ and all $n\in\Z$ we have
    $(f_n[x],g_n[x])\in G$.
\end{enumerate}

The conditions defining $\mathrm{mark}_d$ express that the atlas
encoded by $(g_n)$ is simply the atlas encoded by $(f_n)$ composed with
a fixed homeomorphism of $M$ whose graph is encoded by $G$. For
$(M_0,G_0)\in \mathrm{HGrp}_d$ fixed, we let $\mathrm{mark}_d(M_0,G_0)$
consist of pairs $((f_n),(g_n))\in\mathfrak m_d\times \mathfrak m_d$
which each encode $M_0$, and where for all $x\in B$ and all $n\in\Z$ we
require $(f_n[x],g_n[x])\in G_0$. The following is straightforward:

\begin{prop}\label{prop:mark-projective}
    The set $\mathrm{mark}_d$ is parameter-free projective.
\end{prop}


\subsection{Uniform interpretations and first-order rigidity}
\newcommand{\M}[0]{\mathcal{M}}
\newcommand{\F}[0]{\mathcal{F}}
\newcommand{\A}[0]{\mathcal{A}} 
\newcommand{\B}[0]{\mathcal{B}}

\begin{defn}\label{def:rigidity}
	Let $\F$ be a family of structures in a countable signature $\mathcal{S}$
    that is
    definable in second-order arithmetic. We write $\For(\F)$ for the set theoretical statement that expresses the fact that the class $\F$ is first-order rigid, that is that $\A\equiv\B$ implies $\A\cong\B$ for $\A,\B\in \F$. This is expressed by a formula on the same parameters used to define $\mathcal{F}$.
\end{defn}

This work can be understood as an inquiry into the set theoretical strength of statements of the form $\For(\F)$ where $\F$ is one of the following:
\begin{itemize}
	\item the collection of all homeomorphism groups of manifolds of a fixed dimension $m$, 
	\item for a fixed compact manifold $M$, the collection of all expansions of the group structure $\Homeo(M)$ by a constant  that distinguishes an element of $\Homeo(M)$, or some variation thereof.
\end{itemize}

\begin{defn}
	Consider families of structures $\F$ and $\F'$, with countable
    signatures $\mathcal{S}$ and
    \[\mathcal{S}'=\{R_i^{k_i},f_j^{r_j}\mid i\in I,j\in J\},\]
    respectively. By a \emph{uniform interpretation $\iota$ of structures
    from $\F'$ in $\F$} we mean a collection of parameter-free
    $\mathcal{S}$-formulae
    \[\theta,\chi, \{\phi_i\}_{i\in I},\{\psi_j\}_{j\in J}\]
    such that for every $\A\in\F$ these formulae interpret a structure
    $\iota(\A)\in\F'$ in the usual sense; that is:
	\begin{itemize}
		\item $\theta$ defines a subset $X\subseteq \A^k$, and $\chi$ defines
        an equivalence relation $E$ on $X$;
		\item for every $i\in I$, the relation defined by $\phi_i$ factors
        through a $k_i$-ary predicate on $X/E$;
		\item for every $j\in J$, the relation defined by $\psi_j$ factors
        through the graph of an $r_j$-ary function on $X/E$;
		\item $\iota(\A)$ is the structure with universe $X/E$, where $R_i$ is
        interpreted by $\phi_i$ and the function symbol $f_j$ by $\psi_j$.
	\end{itemize}
  In this situation we may also write, more succinctly, that $\A\in\F$
  uniformly interprets $\iota(\A)\in\F'$ without parameters. If
  $\iota(\A)\cong\iota(\B)$ always implies $\A\cong\B$, then we say that
  $\iota$ is \emph{faithful}, and we write $\F\leq^{\mathrm{ufi}}\F'$.
\end{defn}

It is a standard fact (i.e.~a consequence of ZFC) that if $\A\equiv\B$, then any parameter-free uniform interpretation gives $\iota(\A)\equiv\iota(\B)$.
It immediately follows that the result below is a theorem of ZFC:
\begin{lemma}\label{l:uniform intepretation and rigidity}
	For $\F,\F'$ as in Definition \ref{def:rigidity} we have: 
	 \begin{equation*}
	 	\F\leq^{\mathrm{ufi}}\F'\,\,\longrightarrow\,\,\left (\For(\F')\rightarrow\For(\F)\right).
	 \end{equation*}
\end{lemma}


\section{First-order properties of homeomorphism groups and topological preliminaries}

In this section, we gather some background material about first-order rigidity
of homeomorphism groups of manifolds, as well as some topological tools adapted
to the study of noncompact manifolds embedded in Euclidean space.

\subsection{First-order theory of homeomorphism groups of manifolds}

Recall that all manifolds $M$ are
assumed to be connected, second countable, and Hausdorff, though not necessarily compact. To avoid some technical difficulties, we always assume $\partial M=\varnothing$. The second countability of $M$ is equivalent to $\sigma$--compactness of $M$, so that $M$ admits a countable exhaustion by compact submanifolds.

For compact manifolds $M$ (possibly with boundary), the first-order theory of
the homeomorphism group $\Homeo(M)$ in the language $\mL$
was developed extensively
in~\cite{KKdlN22}. Most importantly for our purposes, the content of
~\cite{KKdlN22} is that homeomorphism groups of compact
manifolds uniformly
interpret many auxiliary sorts together with definable predicates
relating them to each other and to the home sort, which are useful for
establishing rigidity results.

Suppose $M$ is a manifold of some dimension $d\geqslant 1$.
A homeomorphic embedding $\phi\colon B^d(1)\rightarrow M$ that extends to an embedding  $B^d(2)\to M$ is called a \emph{parametrized collared ball} in $M$;
we will simply call the image of $\phi$ a \emph{collared ball}.

A subset of a Polish space is definable (with parameters, generally)
in second-order arithmetic
if and only if it is projective, in the sense of descriptive set theory;
see~\cite{simpson-book}.

We will assume all abstract topological spaces to be at least Hausdorff.
In a topological space, an open set is called \emph{regular} if it is
the interior of its closure. The regular open sets of a topological
space $X$ are written $\ro(X)$, and form a Boolean algebra.
A \emph{regular compact set} is a compact set whose complement is
regular open. A compact set is regular if it is the closure of its
interior. The regular compact subsets of $X$ are written $\rc(X)$.

The following result summarizes much of the technical content of
~\cite{KKdlN22}. We briefly fix terminology. By an \emph{imaginary sort}
of a structure $A$ we mean a definable quotient $X/E$, where
$X\subseteq A^n$ is definable and $E$ is a definable equivalence relation
on $X$. For each imaginary sort listed below, we say that the sort is
\emph{uniformly interpretable} in $\Homeo(M)$ if there are formulae in the
language $\mL$ defining $X$, $E$, and all relevant relations, and these
formulae do not depend on $M$; that is, the manifold $M$ may be treated as a
variable, and the same formulae define the corresponding sort for each
manifold. All definitions and interpretations are parameter-free, unless
stated otherwise. The term uniform(ly), present in isolation, is always to be
understood as ``independent(ly) of the manifold $M$". 

\begin{thm}[See~\cite{KKdlN22}]\label{thm:kkdln}
    Let $M$ be an arbitrary compact, connected manifold,
    possibly with boundary, and let $\Homeo_0(M)\leqslant G\leqslant\Homeo(M)$. Several collections of objects derived from $M$ can be shown to  admit a canonical bijection with a series of imaginary sorts in $G$, as a structure in the language of groups. We also claim that certain predicates defined over said derived data are definable, i.e., definable modulo the canonical identification in question.  
    \begin{enumerate}
        \item There is a uniform imaginary sort in canonical bijection with the collection $\ro(M)$ of regular open
        subsets of $M$. Moreover, there are uniformly definable predicates encoding:
        \begin{enumerate}
            \item The Boolean algebra structure on $\ro(M)$;
            \item The action structure on $\ro(M)$; that is, the graph of the map $act:G\times \ro(M)\longrightarrow \ro(M)$ mapping
            $(g,U)$ to $g\cdot U$. 
        \end{enumerate}
        \item There is a uniform imaginary sort admitting a canonical
        bijection with the natural numbers $\N$, and another one admitting a canonical bijection with the collection $2^{\N}$ of subsets of $\N$, so that the following predicates in the sorts above are uniformly definable:
        \begin{enumerate}
            \item the graph of the arithmetic operations $+,\times$ on
            $\mathbb{N}$;
            \item the order on $\mathbb{N}$;
            \item the membership predicate $\in\; \subseteq\mathbb{N}\times 2^{\mathbb{N}}$;
            \item the collection of $(n,U)\in\mathbb{N}\times \ro(M)$ for which $U$ has $n$ connected components.
        \end{enumerate}
        \item There is a uniform imaginary sort in canonical bijection with the set of points of $M$ with respect to which the membership predicate $[\mathbf{\in}']\subseteq M\times \ro(M)$ and the predicate of the action of $G$ on $M$ are uniformly definable. 
        \item There is a uniform imaginary sort
        whose elements are in canonical bijection with the collection $\mathcal P^{<\omega}(M)$ of finite sequences of points in $M$. Moreover, the predicate $\in(\pi,k,U)$, that holds precisely when the $k^{th}$ term of
        $\pi\in\mathcal P^{<\omega}(M)$ lies in $U\in\ro(M)$ is uniformly definable.
        \item 
        The collection of tuples $(g,n,U,V)\in G\times\mathbb{N}\times \ro(M)^{2}$
        for which $g^n(U)=
        V$ is uniformly definable. The same holds if
         $U$ and $V$ are replaced by
        points in $M$.
        \item The set of pairs $(p,U)\in M\times\ro(M)$ for which $p\in M$ lies in the closure
        of $U\in\ro(M)$ is uniformly definable.
        \item For each $d\in\N$, there is a sentence $\dim_d$ such that
        $G\models\dim_d$ if and only if $M$ has dimension $d$.
        \item For each fixed dimension $d$, there is an imaginary sort $\rc(M)$, uniform on all $d$-dimensional manifolds $M$,
        in canonical bijection with the regular compact
        subsets of $M$. Moreover, via this bijection, the action
        structure on $\rc(M)$ is given by uniformly definable 
        predicates.
        \item For each fixed dimension $d$, there is an imaginary sort $\mathrm{ball}_{s}(M)$, uniform on all $d$-dimensional manifolds $M$, which
        is in canonical bijection with parametrized collared balls
        $\beta\colon B^d(1)\rightarrow M$. Moreover, the following predicates are uniformly definable for fixed $d$: 
        \begin{enumerate}
            \item The collection $\mathrm{open-ball}\subseteq\ro(M)$ of all $U$ which are the interior of some
            collared ball $\beta(B)$ in $M$.
            \item The collection $\mathrm{closed-ball}(C)\subseteq\rc(M)$ of sets of the form $\beta(B)$ for some
            collared ball in $M$.
            \item The collection $\mathrm{param}$ of all $(\beta,C,q,p)\in\mathrm{ball}_{s}\times\rc(M)\times\mathcal{P}(M)^{2}$ such that:
            \begin{enumerate}
                \item $C=\beta(B)$;
                \item $q\in B\cap \Q^d$, or more generally $q$ is 
                projectively definable in $B$;
                \item $p\in C$;
                \item $\beta(q)=p$.
            \end{enumerate}
        \end{enumerate}
    \end{enumerate}
\end{thm}

Some remarks are in order. First, in the proofs of the items
in Theorem~\ref{thm:kkdln} as given in~\cite{KKdlN2025}, many of the interpretations rely essentially
on the compactness of the underlying manifold $M$, though many do not. 
For the regular open
sets $\ro(M)$ and the action of $\Homeo(M)$ on
$\ro(M)$, the interpretation
can be carried out in a much more general context of locally moving actions
(see~\cite{rubin-short,Rubin1989,Rubin1996}),
and so there is no difficulty replacing $M$ by a
noncompact manifold for interpreting regular open sets. It follows
that regular closed sets are immediately interpretable in
$\Homeo(M)$. The interpretation of $\N$ and the attendant predicates
can be carried out in
any compact set with compact closure in $M$, even when $M$ is compact.
Similar considerations apply to
the interpretations of points. For dimension, since $M$ is connected,
we have that the dimension of $M$ coincides with the dimension of any
nonempty open subset of $M$. In particular, one can express the dimension
of a manifold from a nonempty open set $U$ contained in a collared open
ball; a similar consideration applies to parametrized collared balls in
$M$.

For interpreting regular compact sets when $M$ is
not necessarily compact,
some further steps are necessary. The following proposition is easy,
and follows from straightforward dimension theory of manifolds and
transitivity of the action of the homeomorphism group of a manifold
on collared balls;
see~\cite{KKdlN22,KdlN2023,KdlN2024,LeRoux2014}:
\begin{prop}\label{prop:compact-submanifold}
    Let $M$ be an arbitrary connected manifold of dimension $d$,
    let $\varnothing\neq U\in\ro(M)$, and
    let $K\subseteq M$ be arbitrary.
    \begin{enumerate}
        \item The closure of $U$ lies inside a collared ball
        in $M$ if and only if, for all nonempty $W\in\ro(M)$,
        there is
        a $g\in\Homeo(M)$ such that $g(U)\subseteq W$.
        \item The subset $K$ lies in a compact submanifold of $M$ if
        and only if there exists a finite collection
        $\{U_1,\ldots,U_m\}\subseteq \ro(M)$ such that each point
        $p\in K$ lies in some $U_i$, and such that each $U_i$
        has closure contained in a collared ball in $M$.
        \item Conversely, if $N\subseteq M$ is a compact,
        connected submanifold,
        then there is an $m\in\N$
        depending only on the dimension $d$ of
        $M$ such that $N$ is contained in the union of regular open sets
        $\{U_1,\ldots,U_m\}\subseteq \ro(M)$, each of which has
        closure contained in a collared ball in $M$.
    \end{enumerate}
\end{prop}

It follows from Proposition~\ref{prop:compact-submanifold} that there is 
a uniform imaginary sort in canonical bijection with  $\rc(M)$ so that the relevant predicates are all uniformly definable
in $\Homeo(M)$ by $\mL$--formulae, provided that the dimension of $M$
is a fixed value $d$.

The uniform interpretation of $\N$ and the related predicates
means that definably presented Polish spaces, the Borel hierarchy, and the projective
hierarchy, are all uniformly interpretable across homeomorphism
groups of manifolds. Thus, all homeomorphism groups of
manifolds ``agree" on what $B$ is, what Euclidean spaces are,
what continuous functions are, etc.

\begin{cor}\label{cor:noncompact-interp}
    The conclusions of Theorem~\ref{thm:kkdln} hold for arbitrary
    connected manifolds.
\end{cor}

A further sort which is uniformly interpretable for manifolds of fixed
dimension $d$ is that of sequences of points; this interpretation
was carried out in~\cite{KdlN2023}:

\begin{prop}\label{prop:interp-seq}
    Let $d$ be fixed and $M$ be a compact manifold of dimension $d$. Given a group \[\Homeo_0(M)\leqslant G\leqslant\Homeo(M),\] there is a uniform imaginary sort $\mathrm{seq}(M)$, which admits a canonical bijection with the collection of countable sequences of points in $M$. Moreover, for fixed $d$ the set of pairs \[(p,k,\sigma)\in M\times \N\times \mathrm{seq}(M)\] for which $p$ is the $k^{th}$ term of $\sigma$ is uniformly definable.
\end{prop}

Proposition~\ref{prop:interp-seq} is generalized to noncompact
manifolds below in Corollary~\ref{cor:noncompact-sequence}.

Let $X$ be a fixed separable Hausdorff
topological space that is parameter-free interpretable
in second-order arithmetic; that is to say, the points
$\mathcal P(X)$ and open sets $\mathcal U$
can be canonically identified with imaginary sorts in second-order arithmetic, so that the inclusion relation
$p\in U$
of a point in an open set is uniformly definable; for instance, $X$ may be 
a definably presented Polish space. Then, homeomorphisms of $X$ are determined
by their values on a dense subset of $X$, and countable sequences of
points in $X$ are a definable sort in second-order arithmetic.
Combining with Proposition~\ref{prop:interp-seq}, there is a uniformly
definable sort $\mathrm{seq}(M\times X)$ in canonical bijection with
countable sequences in $M\times X$, together with a uniformly definable
predicate defining the $k^{th}$ term of a sequence, for all
$k\in \N$.

It is not difficult, from direct access to regular open sets in $M$
and the topology of $X$, to construct definable predicates which
define sequences in $M\times X$ encoding graphs of homeomorphisms
of $M\times X$. This was done explicitly for compact manifolds in~\cite{KdlN2023},
and we spell out some details in Corollary~\ref{cor:noncompact-sequence} below
in the case of general connected boundaryless manifolds.

\subsection{Dynamics of homeomorphisms of noncompact manifolds}

Crucial to our investigation of noncompact manifolds and encoding
such manifolds with finite data is the notion of a \emph{topologically
transitive} homeomorphism. A homeomorphism $g$ of a topological space
$X$ is topologically transitive if for all nonempty open $U,V\subseteq X$,
there exists an $n\in\Z$ such that $g^n(U)\cap V\neq\varnothing$.
Clearly there are no topologically transitive homeomorphisms of $\R$,
though in dimension two or more, topologically transitive homeomorphisms
always exist (see~\cite{alpern-prasad}):
\begin{thm}\label{thm:top-trans}
    Let $M$ be a connected $\sigma$--compact
    manifold of dimension at least two. Then $M$ admits a topologically
    transitive homeomorphism.
\end{thm}

For us, topologically transitive homeomorphisms are a parameter-free
definable subset of $\Homeo(M)$.

\begin{prop}\label{prop:top-trans-def}
    The set $\mathcal T\subseteq\Homeo(M)$ consisting of topologically
    transitive homeomorphisms is uniformly definable in the language
    $\mL$ without parameters.
\end{prop}
\begin{proof}
    We simply define $\mathcal T$ to consist of all $g\in\Homeo(M)$ such
    that for all nonempty $U,V\in\ro(M)$, there exists a $p\in U$ and an
    $n\in\Z$ such that $g^n(p)\in V$. That these conditions are first-order
    expressible follows from Theorem~\ref{thm:kkdln} and
    Corollary~\ref{cor:noncompact-interp}.
\end{proof}

For us, topologically transitive homeomorphisms of a manifold $M$ are
merely a means to an end. For an open set $U\subseteq M$ and a
homeomorphism $g\in \Homeo(M)$, we write
\[M_{g,U}=\bigcup_{i\in\Z}g^i(U).\]
If $M_{g,U}$ is a dense subset of
$M$ then we say that $M_{g,U}$ is a \emph{$g$--dense} submanifold of
$M$ (or just a \emph{dense} submanifold if $g$ and $U$ are not relevant
for the discussion), and we say that $g$ is 
\emph{sufficiently transitive}.
We have the following, which can be proved the same way
as Proposition~\ref{prop:top-trans-def}:

\begin{prop}\label{prop:g-dense}
    The set $\mathcal G(M)$ consisting of all $(g,U)\in \Homeo(M)\times\ro(M)$ such that
\[\begin{split}
    U \text{ is the interior of a collared ball} \quad\&\quad M_{g,U} \text{ is $g$-dense}
\end{split}\]
    is uniformly definable without parameters. Furthermore, for all connected, boundaryless manifolds $M$, we have $\mathcal G(M)\neq \emptyset$.
\end{prop}

For Proposition~\ref{prop:g-dense}, a word needs to be said in dimension one, since $\R$ admits no topologically transitive homeomorphisms. To establish the nontriviality of $\mathcal G(M)$ for the cases $M=\R$ and $M=S^1$, we proceed as follows. Taking $U=(-1/2, 1)\subseteq\R$ and $g\colon x\mapsto x+1$ furnishes an element of $\mathcal G(M)$. For $S^1$, we simply take a nonempty arc whose closure is not $S^1$ and let $g$ be an irrational rotation.

Proposition~\ref{prop:g-dense} allows us to generalize 
Proposition~\ref{prop:interp-seq} to connected, boundaryless
manifolds that are not necessarily compact.

\begin{cor}\label{cor:noncompact-sequence}
    Let $M$ be a connected, boundaryless manifold of dimension $d$.
    There exists a uniform imaginary sort in $\Homeo(M)$
    admitting a canonical bijection with
    the collection $\mathrm{seq}^{\mathcal G}(M)$ of countable sequences $\sigma$ of points in $M$, and a uniformly definable predicate on
    triples $(p,k,\sigma)$ expressing that $p$ is the $k^{th}$ entry of $\sigma$.
\end{cor}
\begin{proof}
	  We first claim that there exists a uniformly definable sort
    $\mathrm{seq}^{\mathcal G}(M)$, in canonical bijection with
    countable sequences $\sigma$ of points in $M$ for which there
    is a $(g,U)\in\mathcal G(M)$ such that 
    \[\sigma\subseteq M_{g,U}\subseteq M,\]
	  together with the corresponding uniformly definable entry extraction predicate.
	 
    Fix a pair $(g,U)\in\mathcal G(M)$, with $B$ denoting the closure
    of $U$. An identical argument as for establishing 
    Proposition~\ref{prop:interp-seq} in~\cite{KdlN2023} shows that
    the imaginary sort of countable sequences in $B$ is uniformly
    $\mL$--interpretable, with $(g,U)$ as a parameter. If
    $\sigma$ is a sequence of points in $B$ and
    $\tau\in\Z^{\N}$
    then we obtain a sequence $\xi(\sigma,\tau)$ in $M_{g,U}$ by
    specifying its $n^{th}$ term for all $n$:
    $\xi(n)=g^{\tau(n)}(\sigma(n))$. It is clear that every sequence
    of points in $M_{g,U}$ arises as $\xi(\sigma,\tau)$ for some
    $\sigma$ and $\tau$, and so $\mathrm{seq}^{\mathcal G}(M)$ can be
    defined by 
    $$
    \MGD{\xi}
    { \exists g\,\exists U\;\exists\sigma\;\exists\tau\;
   \Big( 
   (g,U)\in\mathcal G(M)     \wedge  \forall n\; \xi(n)=g^{\tau(n)}\big(\sigma(n)\big)
   \Big)
    }
    $$
    That the entry extraction predicate is uniformly definable is trivial.
    
    To conclude the proof of the Lemma, recall the standard fact that there is a bijection $(m,n)\mapsto\langle m,n\rangle$ between $\N^{2}$ and $\N$ definable without parameters in arithmetic. Now, consider the collection $\mathrm{seq}^{\mathcal G^{*}}(M)$ of sequences $\sigma\in\mathrm{seq}^{\mathcal G}(M)$ such that for each $m\in\N$ the sequence $\sigma(\langle m,n\rangle)$ converges to a point $p_{m}\in M$. It is trivial from all the definability results presented so far that 
    this collection is uniformly definable and that so is the graph of the map 
    \[\N\times\mathrm{seq}^{\mathcal G^{*}}(M)\longrightarrow M,\quad\quad (m,\sigma)\mapsto \lim_{n}\sigma(\langle m,n\rangle).\]
    Clearly, for all $\tau\in\mathrm{seq}^{\mathcal G}(M)$ there is $\sigma\in\mathrm{seq}^{\mathcal G^{*}}(M)$ such that 
    $\tau(m)=\lim_{n}\sigma(\langle m,n \rangle)$ for all $m$, so it can be easily concluded that $\mathrm{seq}^{\mathcal G}(M)$ can be identified with a suitable quotient of $\mathrm{seq}^{\mathcal G^{*}}(M)$ by a uniformly definable equivalence relation. 
\end{proof}

From this point the same arguments as in \cite{KdlN2023} one can conclude the following:   
\begin{cor}\label{sequentially complete}
	Fix a dimension $d\geqslant 1$. For all imaginary sorts $\mathrm{S}$ in the structure $\Homeo(M)$, there exists another sort, defined by formulae that are
    uniform in all boundaryless manifolds $M$ of dimension $d$,
    which encodes the collection $\mathrm{seq}^{S}(M)$ of sequences of elements of $S$. Moreover, the collection of triples $(e,k,\sigma)\in\mathrm{S}\times\N\times\mathrm{seq}^{S}(M)$ such that $\sigma(k)=e$ is uniformly definable for fixed $d$.   
\end{cor}

For $X$ a separable Hausdorff topological space that is parameter-free definable in second-order arithmetic, we write 
$\mathrm{seq}^{\mathcal G}(M\times X)$ for sequences in $M\times X$
where the projection to the first coordinate is an element of
$\mathrm{seq}^{\mathcal G}(M)$.

\begin{cor}\label{cor:noncompact-inner}
    Let $M$ be a connected, boundaryless manifold of dimension $d$
    and let
    $X$ be a separable Hausdorff topological space that is parameter-free definable in second-order arithmetic.
    There is a uniformly definable imaginary sort in $\mathrm{seq}^{\mathcal G}
    (M\times X)$
    which
    is in canonical bijection with elements of
    $\Homeo(M\times X)$. Moreover,
    the group operation on $\Homeo(M\times X)$
    is given by a definable predicate.
\end{cor}

\subsection{Language expansion and dimension jumping}

Now, consider the class of connected manifolds up to homeomorphism,
and let $M$ be a manifold of fixed dimension $d$.
By Theorem~\ref{thm:kkdln}, there is an
$\mL$--interpretation in $\Homeo(M)$ of
points $p$ in $M$ and regular open sets $U$ of $M$, together with the
predicate $p\in U$; these interpretations are parameter-free and
uniform in $M$. We thus uniformly interpret an uncountable structure
which consists of two sorts (points and regular open sets) and
which has an underlying signature consisting of the non-logical symbol
$\in$. We (conservatively) expand this structure to a larger
structure $[M]$ which includes a uniform, parameter-free interpretation
of full second-order arithmetic, i.e.~ the sort
$(\N,2^{\N},\in,+,\times,<)$.
It is clear that $[M]$ contains all the data of the topology
of $M$, and the uniformity of the interpretation implies that if
$\Homeo(M)\equiv\Homeo(M')$ then $[M]\equiv [M']$.

By Corollary~\ref{cor:noncompact-sequence}, the sort
$\mathrm{seq}^{\mathcal G}(M)$ of sequences of points lying in
dense submanifolds is uniformly $\mathcal L$--interpretable.
We write $[M]^+$ for the extended structure consisting of $[M]$ together
with the sort $\mathrm{seq}^{\mathcal G}(M)$. Again, if $\Homeo(M)\equiv\Homeo(N)$
then $[M]^+\equiv [N]^+$.

Now, for a fixed $k$, we fix a description of the $k$--sphere $S^k$
as a definably presented Polish space.
By Corollary~\ref{cor:noncompact-inner},
we see the following:

\begin{prop}\label{prop:homeo-interp}
    The structure $[M]^+$ parameter-free interprets the structure
    $\Homeo(M\times S^k)$, uniformly for manifolds $M$ of fixed
    dimension $d$, together with a predicate expressing
    the group operation.
\end{prop}

Because the interpretation of $\Homeo(M\times S^k)$ is parameter-free
and uniform, we have $[M]^+\equiv [N]^+$ implies that
$\Homeo(M\times S^k)\equiv \Homeo(N\times S^k)$. We thus
obtain Theorem~\ref{thm:pd-dim-implication} from the introduction
as an immediate corollary of Proposition~\ref{prop:homeo-interp}.

\subsection{Noncompact surfaces and ordinals}

From Corollary~\ref{sequentially complete} above, we have the following fact relating
the topology of subspaces of the Cantor space $2^{\N}$ to first-order rigidity:

\begin{cor}\label{space of ends}
    Let $\alpha<\omega_{1}$ be an ordinal, let $C_\alpha\subseteq S^2$ be a closed set homeomorphic to $\omega^\alpha+1$ and set $S=S^2\setminus C_\alpha$. 
    Let $\mathrm{COrd}$ be a uniformly definable sort in $\Homeo(S)$
    that interprets the collection of all countable ordinals. Then there is
    a formula $\phi(x)$, independent of the underlying surface,
    such that $\Homeo(S)\models \phi(\beta)$ for a countable ordinal $\beta$ if and only if $\beta=\alpha$. 
\end{cor}
\begin{proof}
By applying Corollary \ref{sequentially complete} one can successively encode
the following data as imaginary classes, uniformly in all manifolds of dimension $d$: 
\begin{itemize}
	\item open subsets of $M$, seen as sequences of interiors of embedded balls quotiented by the equivalence relation of having the same union;
	\item the binary relation which holds between open sets when one is a connected component of the other; 
	\item compact subsets of $M$;
	\item infinite sequences $\big((K_{n},U_{n})\big)_{n\in\mathbb{N}}$, where $K_{1}\subseteq K_{2}\dots$ is an exhaustion of $M$ by compact sets and $U_{n}$ is a connected component of $M\setminus K_{n}$; 
	\item the equivalence relation between the sequences \[\big((K_{n},U_{n})\big)_{n\in\mathbb{N}}\quad \textrm{and}\quad \big((K'_{n},U'_{n})\big)_{n\in\mathbb{N}}\] given by $(U_{n})_{n\in\mathbb{N}}$ and $(U'_{n})_{n\in\mathbb{N}}$ are mutually coinitial for the inclusion, and whose equivalence classes are in bijection with the space of ends;  
	\item the relation between open sets $U\subseteq M$ and ends $E$ of $M$, 
    given by $E\in\bar{U}$. 
\end{itemize}
 Countable ordinals exist as an imaginary sort $S$ in second-order arithmetic, and the order topology on any such ordinal is definable in a way that does not depend on the ordinal. It is thus trivial to write down a formula without parameters $\phi(\alpha)$ which expresses the existence of a homeomorphism between $\omega^\alpha+1$ and the space of ends. 
\end{proof}

\begin{cor}
	Suppose that all elements of the group $\Homeo_0([0,1])$ are type rigid and that $\alpha\neq\alpha'$ are countable ordinals. Then $\Homeo(S_{\alpha})\not\equiv\Homeo(S_{\alpha'})$.
\end{cor}
\begin{proof}
	It is not hard to check that in second-order arithmetic, there is a map
    that is definable without parameters, which to each countable ordinal $\alpha$ assigns an element $g(\alpha)\in\Homeo_0([0,1])$ in such a way that $g(\alpha)$ 
    and $g(\alpha')$ are conjugate in $\Homeo([0,1])$ only if $\alpha=\alpha'$.
    Such a map can be obtained, for instance, by blowing up some canonical
    (up to an orientation preserving homeomorphism) embedding of the ordinal into $[0,1]$. Corollary \ref{space of ends} and Lemma \ref{l:uniform intepretation and rigidity} give the desired conclusion.
\end{proof}

\subsection{A topological fact}

In this section, we record a fact from
point-set topology, which will be useful in the sequel.

\begin{prop}\label{prop:top}
    Let $X$ and $Y$ be separable, locally compact, complete metric spaces. Assume that  $\{x_n\}_{n\in\N}\subseteq X$ and
    $\{y_n\}_{n\in\N}\subseteq Y$ are dense subsets of $X$, respectively of $Y$ such that:
    \begin{enumerate}
        \item for each regular compact $C\subseteq X$ there exists
        a regular compact $K\subseteq Y$ such that
        \[x_n\in C \Leftrightarrow y_n\in K;\]
        \item for each regular compact $K\subseteq Y$ there exists
        a regular compact $C\subseteq X$ such that
        \[x_n\in C \Leftrightarrow y_n\in K.\]
    \end{enumerate}
    Then the map $x_n\mapsto y_n$ extends
    to a homeomorphism \[\theta\colon X\longrightarrow Y.\]
\end{prop}
\begin{proof}
    Because the set $\{x_n\}_{n\in\N}$ is dense in $X$, we see that a regular closed set $C\subseteq X$ is completely determined by the intersection $C\cap \{x_n\}_{n\in\N}$. Similarly for regular closed subsets of $Y$.
    
    Let $x\in X$ be any given point and find, by local compactness, a sequence $(C_m)_{m\in \N}$ of regular
    compact subsets of $X$ such that $\{x\}=\bigcap_m C_m$.
    We let
    $K_m$ be the (unique) regular compact set in $Y$ corresponding via (1) to $C_m$
    and we consider $\bigcap_m K_m$. Because the $K_m$ are nested,
    compact, and nonempty, we have that this intersection contains at
    least one point $y$. Suppose for a contradiction that $y'\neq y$ is another point in this
    intersection and choose disjoint regular compact neighborhoods $K\ni y$ and $K'\ni y'$. Set 
    $$
    L_m=\overline{{\sf int}\big(K_m\cap K\big)}\qquad\&\qquad L'_m=\overline{{\sf int}\big(K_m\cap K'\big)}
    $$
    and note that $(L_m)_{m\in \N}$ and $(L'_m)_{m\in \N}$ form disjoint  nested sequences of regular compact subsets with 
    $$
    y\in \bigcap_mL_m \qquad\&\qquad y'\in \bigcap_mL'_m.
    $$
    The $L_m$ and $L'_m$ correspond via (2) to regular compact sets 
    $N_m\subseteq C_m$ and $N_m'\subseteq C_m$ such that $N_m\cap N_m'=\emptyset$. Moreover, by compactness again, we have
    $$
    \emptyset\neq \bigcap_mN_m\subseteq  \bigcap_m C_m =\{x\}  
    $$
    and
    $$
    \emptyset\neq \bigcap_mN'_m\subseteq  \bigcap_m C_m =\{x\},  
    $$
    which is absurd.

    Thus, $\bigcap_n K_n$ consists of exactly the one point $y$. We define $\theta(x)=y$. It can be
    verified then that a sequence in $\{x_n\}_{n\in\N}$ converges
    to a point $x\in X$ if and only if the sequence in $\{y_n\}_{n\in \N}$
    with the same indices converges to a point $y\in Y$, and
    $\theta(x)=y$. This implies that $\theta$ is a continuous bijection with a
    continuous inverse.
\end{proof}



\section{Constructibility and rigidity}

Let $M$ be a connected, abstract manifold. To keep a succinct notation,
in this section, we set $G_M=\Homeo(M)$ and let $g\in G_M$.
We wish to show that under {V=L}, the group $G_M$ is first-order rigid
and the homeomorphism $g$ is type rigid. All interpretations of structures in
$G_M$ are over the language $\mL$.

We will use the following standard consequences of {V=L}. First, the
canonical well-ordering $<_{L}$ of $L$ is definable without parameters in the
language of set theory. Second, if $X$ is a definably presented Polish space,
then $<_{L}$ induces a parameter-free projective well-ordering of $X$. Third,
every second countable manifold admits a real code; hence, under {V=L}, every
such manifold is homeomorphic to one coded in $L_{\omega_1}$.

\begin{prop}[See~\cite{moschovakis-book}]
    Assume {V=L}, and let $X$ be a definably presented
    Polish space. Then there is a parameter-free projective well-ordering on $X$.
\end{prop}

\subsection{First-order rigidity}\label{ss:vl-fo-rigid}
For this section, let $M$ be fixed of dimension $d$,
and let $N$ be an arbitrary
connected manifold with $G_N=\Homeo(N)$.

\begin{thm}\label{thm:vl-forigid-body}
    Assume {V=L}. Then $G_M\equiv G_N$ if and only if $M$ and $N$ are
    homeomorphic manifolds.
\end{thm}

Obviously we need only prove that if $G_M\equiv G_N$ then
$M\cong N$.
The first reduction we can make is to assume that $N$ also has
dimension $d$, as follows from Theorem~\ref{thm:kkdln}.

In general,
we will need to build a nexus between the first-order theory
of $G_M$ and the descriptive set theory of $C(B,E)^\Z$; this will
be provided almost entirely by Theorem~\ref{thm:kkdln} and
Theorem~\ref{thm:top-trans}.
We begin by fixing a uniform interpretation of second-order arithmetic,
which allows for direct access to all definably presented
Polish spaces, including $B=B^d(1)$ and
$C(B,E)^\Z$. We fix a uniform interpretation of $\mathrm{ball}(M)$,
and we let $\mathcal G(M)$ denote the (uniformly definable) set of pairs $(\tau,U)$ encoding
sufficiently transitive homeomorphisms of $M$ and $\tau$--dense submanifolds of $M$. We fix a $\beta\in
\mathrm{ball}(M)$ and $(\tau,\beta)\in \mathcal G(M)$ such that
\[M(\tau,\beta)=\bigcup_{n\in\Z} \tau^n\beta[B]\subseteq M\] is dense.
We let
$B_{\Q}\subseteq B$ denote the set of rational points (which
may be viewed as the countable subset giving the effective
presentation of $B$), and we set
\[D=D(\beta,\tau)=\bigcup_{n\in\Z} \tau^n\beta[B_{\Q}]\subseteq M.\]
By fixing a definable enumeration of $\Q^d\times\Z$, we obtain a fixed
enumeration of $D=\{d_n\}_{n\in\N}$.

From the data of $\tau$ and $\beta$, we obtain a (parameter-free
definable) family of elements
$\mathcal D(\tau,\beta)\subseteq C(B,E)^\Z$ consisting of sequences
$(f_n)$ such that:
\begin{enumerate}
    \item We have $(f_n)\in\mathfrak m_d$;
    \item For all $x,y\in B$ and all $i,j\in\Z$,
    we have $f_i(x)=f_j(y)$ if and only if 
    $\tau^i\beta[x]=\tau^j\beta[y]$.
\end{enumerate}

In particular, the set $\mathcal D(\tau,\beta)\subseteq C(B,E)^\Z$ 
consists of
sequences $(f_n)$ for which $\bigcup_{n\in\Z} f_n[B]$ is homeomorphic to
$M(\tau,\beta)$.

For $(f_n)\in C(B,E)^\Z$, we let \[D_{(f_n)}=\bigcup_{n\in\Z} f_n(B_{\Q}).\]
From the fixed enumeration of $\Q^d\times\Z$, we obtain an enumeration
$D_{(f_n)}=\{e_n\}_{n\in\N}$.

\begin{lem}\label{lem:rec-define}
    There is a nonempty $\mL$--definable subset $\mathcal M(\tau,\beta)\subseteq
    \mathcal D(\tau,\beta)$ consisting of sequences $(f_n)$ such that
    \[\overline{\bigcup_{n\in\Z} f_n[B]}\cong M.\]
\end{lem}
\begin{proof}
    By Lemma~\ref{lem:locally-flat-euclidean-embedding}, the manifold $M$ embeds in $E$,
    and therefore such a sequence $(f_n)$ exists. Because the class of
    regular compact subsets of $E$ is parameter-free projective and because the
    sort of regular compact subsets of $M$ is uniformly $\mL$--interpretable,
    we may require that for every regular compact subset 
    \[C\subseteq \overline{\bigcup_{n\in\Z} f_n[B]}\cong M,\] there exists
    a regular compact $K\subseteq M$ such that for all $n$,
    we have $e_n\in C$ if and only if $d_n\in K$, and conversely
    switching the roles of $C$ and $K$. Then, the hypotheses of
    Proposition~\ref{prop:top} are satisfied, and the bijection
    $\theta\colon e_n\mapsto d_n$ extends to a homeomorphism 
    \[\Theta\colon \overline{\bigcup_{n\in\Z} f_n[B]}\longrightarrow
    M,\] as desired.
\end{proof}

We set \[\mathcal M(M)=\bigcup_{\tau,\beta}\mathcal M(\tau,\beta),\]
which is clearly definable, since $(f_n)\in\mathcal M(M)$ if and only
if \[\exists\tau\exists\beta\;(f_n)\in\mathcal M(\tau,\beta).\]
We think of $\mathcal M(M)$ as the set of
codes for $M$.
We can now establish first-order rigidity.

\begin{proof}[Proof of Theorem~\ref{thm:vl-forigid-body}]
    Suppose $G_M\equiv G_N$. By the dimension sentence in
    Theorem~\ref{thm:kkdln}, the manifolds $M$ and $N$ have the same
    dimension $d$. Assume {V=L}, and fix the parameter-free projective
    well-ordering $<_{L}$ of the code space $C(B,E)^\Z$ induced by the
    canonical well-ordering of $L$.

    For any connected $d$-manifold $P$, let $c(P)$ denote the
    $<_{L}$-least element of the interpreted set $\mathcal M(P)$ of codes
    for $P$. This is meaningful because $\mathcal M(P)$ is nonempty by
    Lemma~\ref{lem:rec-define}, and under {V=L} the code space and the
    definable set $\mathcal M(P)$ are well-ordered by $<_{L}$.

    We claim that $M\cong N$ if and only if $c(M)=c(N)$. If
    $c(M)=c(N)$, then both manifolds are homeomorphic to the same embedded
    manifold coded by this common element. Conversely, if
    $\Phi\colon M\to N$ is a homeomorphism, then conjugating the auxiliary
    data $(\tau,\beta)$ by $\Phi$ shows that $\mathcal M(M)=\mathcal M(N)$;
    hence the $<_{L}$-least elements agree.

    Let $\{\mathcal U_i\}_{i\in\N}$ be the fixed recursive basis for the
    topology of $C(B,E)^\Z$. For each $i$, let $\chi_i$ be the translation
    into the group language of the following assertion in the uniformly
    interpreted code sort:
    \[
    \exists a\left(a\in\mathcal U_i\wedge a\in\mathcal M(\cdot)\wedge
    \forall b\,(b<_{L}a\rightarrow b\notin\mathcal M(\cdot))\right).
    \]
    Thus $<_{L}$ is not added to the group language; it is used only inside
    the uniformly interpreted second-order-arithmetic/code apparatus, and
    the whole displayed assertion translates to an ordinary first-order
    $\mL$-sentence.

    Since $G_M\equiv G_N$, for every $i$ we have
    \[
    G_M\models\chi_i\quad\Longleftrightarrow\quad G_N\models\chi_i.
    \]
    The basis $\{\mathcal U_i\}_{i\in\N}$ separates points in the code
    space, so $c(M)=c(N)$. Therefore $M\cong N$.
\end{proof}

\subsection{Type rigidity}\label{ss:vl-type}

Now, let $(M,g)$ be a fixed pair, where $M$ is an abstract
connected manifold of dimension $d$ and where $g\in\Homeo(M)$.

\begin{thm}\label{thm:vl-type-rigid-body}
    Suppose {V=L} and let $h\in\Homeo(M)$ be arbitrary. Then
    $\tp(g)=\tp(h)$ if and only if $g$ and $h$ are conjugate in
    $\Homeo(M)$.
\end{thm}

We now refine the coding from Subsection~\ref{ss:vl-fo-rigid} in order to
code a homeomorphism together with the underlying manifold. Let
\[a=((f_n),(u_n))\in \mathrm{mark}_d.\]
Write
\[
X_a=\overline{\bigcup_{n\in\Z} f_n[B]}
     =\overline{\bigcup_{n\in\Z} u_n[B]}.
\]
Since the first union is dense in $X_a$, the graph witnessing
$a\in\mathrm{mark}_d$ is unique; we denote by
\[F_a\colon X_a\longrightarrow X_a\]
the corresponding homeomorphism, so that $F_a(f_n(t))=u_n(t)$ for all
$n\in\Z$ and $t\in B$.

Fix once and for all a recursive enumeration
$m\mapsto (i_m,q_m)$ of $\Z\times (B\cap\Q^d)$. If $(\tau,\beta)$ is as in
Subsection~\ref{ss:vl-fo-rigid}, put
\[
d_m^{\tau,\beta}=\tau^{i_m}\beta[q_m]\in M,\qquad
 e_m^a=f_{i_m}(q_m)\in X_a,\qquad
 r_m^a=u_{i_m}(q_m)\in X_a.
\]
Thus $\{d_m^{\tau,\beta}\}_{m\in\N}$ is the dense set in $M$ associated to
$(\tau,\beta)$, while $\{e_m^a\}_{m\in\N}$ is its counterpart in the embedded
model $X_a$.

For $C\in\rc(X_a)$ and $K\in\rc(M)$, say that $C$ and $K$ are
$(a,\tau,\beta)$-matched if
\[
\forall m\in\N\quad\big(e_m^a\in C \Longleftrightarrow
 d_m^{\tau,\beta}\in K\big).
\]
By the definition of $\mathcal M(\tau,\beta)$ and Proposition~\ref{prop:top},
matched regular compact sets are precisely the regular compact sets
corresponding under the homeomorphism
\[\Theta_{a,\tau,\beta}\colon X_a\longrightarrow M\]
which extends $e_m^a\mapsto d_m^{\tau,\beta}$.

For $x\in G_M$, define
\[\mathcal C(\tau,\beta,x)\subseteq C(B,E)^\Z\times C(B,E)^\Z\]
to consist of all $a=((f_n),(u_n))$ such that:
\begin{enumerate}
    \item $(f_n)\in \mathcal M(\tau,\beta)$;
    \item $a\in\mathrm{mark}_d$;
    \item for every $(a,\tau,\beta)$-matched pair
    $C\in\rc(X_a)$ and $K\in\rc(M)$, and every $m\in\N$, one has
    \[
    r_m^a\in C \Longleftrightarrow x\cdot d_m^{\tau,\beta}\in K.
    \]
\end{enumerate}
Finally set
\[\mathcal C(x)=\bigcup_{\tau,\beta}\mathcal C(\tau,\beta,x).\]
The elements of $\mathcal C(x)$ will be called marked codes for $x$.

The relation $a\in\mathcal C(x)$ is uniformly $\mL$-definable with $x$ as a
parameter. Indeed, the conditions $(f_n)\in\mathcal M(\tau,\beta)$,
$a\in\mathrm{mark}_d$, membership in regular compact sets, the action of
$G_M$ on the point sort, and the parametrization of rational points of
collared balls are all uniformly interpretable in $G_M$. Clause (3) merely
expresses, using regular compact sets, that
\[
\Theta_{a,\tau,\beta}F_a(e_m^a)=x\Theta_{a,\tau,\beta}(e_m^a)
\]
for the dense set $\{e_m^a\}_{m\in\N}$. Hence, by continuity,
\[
\Theta_{a,\tau,\beta}F_a\Theta_{a,\tau,\beta}^{-1}=x.
\]

Moreover, $\mathcal C(x)\neq\varnothing$ for every $x\in G_M$. To see this,
choose $(\tau,\beta)$ and an embedding $j\colon M\to E$, and put
\[
f_n=j\circ \tau^n\circ \beta,\qquad
u_n=j\circ x\circ \tau^n\circ \beta.
\]
Then $a=((f_n),(u_n))$ is a marked code for $x$, with $F_a=jxj^{-1}$ and
$\Theta_{a,\tau,\beta}=j^{-1}$.

\begin{prop}\label{prop:marked-codes-conjugacy}
Let $x,y\in G_M$. Then
\[
\mathcal C(x)\cap \mathcal C(y)\neq \varnothing
\]
if and only if $x$ and $y$ are conjugate in $G_M$. In fact, if $x$ and $y$
are conjugate, then $\mathcal C(x)=\mathcal C(y)$.
\end{prop}

\begin{proof}
Suppose first that $a\in\mathcal C(x)\cap\mathcal C(y)$. Choose witnesses
$(\tau_x,\beta_x)$ and $(\tau_y,\beta_y)$, and write
\[
\Theta_x=\Theta_{a,\tau_x,\beta_x},\qquad
\Theta_y=\Theta_{a,\tau_y,\beta_y}.
\]
By the defining property of marked codes,
\[
x=\Theta_xF_a\Theta_x^{-1}
\quad\text{and}\quad
y=\Theta_yF_a\Theta_y^{-1}.
\]
Therefore $\phi=\Theta_y\Theta_x^{-1}\colon M\to M$ is a homeomorphism
satisfying $\phi x\phi^{-1}=y$. Thus $x$ and $y$ are conjugate.

Conversely, suppose $y=\phi x\phi^{-1}$. If
$a\in \mathcal C(\tau,\beta,x)$, then the same marked code $a$ belongs to
\[
\mathcal C(\phi\tau\phi^{-1},\,\phi\circ\beta,\,y),
\]
since the identifying homeomorphism $X_a\to M$ is simply changed from
$\Theta_{a,\tau,\beta}$ to $\phi\Theta_{a,\tau,\beta}$. Hence
$\mathcal C(x)\subseteq\mathcal C(y)$, and the reverse inclusion follows by
applying the same argument to $\phi^{-1}$. Thus $\mathcal C(x)=\mathcal C(y)$.
\end{proof}

\begin{proof}[Proof of Theorem~\ref{thm:vl-type-rigid-body}]
If $g$ and $h$ are conjugate in $G_M$, then $\tp(g)=\tp(h)$, since
conjugation by an element of $G_M$ is an automorphism of the group $G_M$.

Conversely, assume $\tp(g)=\tp(h)$. By {V=L}, fix the parameter-free
projective well-ordering $<_{L}$ of the Polish code space
\[C(B,E)^\Z\times C(B,E)^\Z.\]
For $x\in G_M$, let $c(x)$ denote the $<_{L}$-least element of
$\mathcal C(x)$. This is meaningful because $\mathcal C(x)$ is nonempty.

Let $\{\mathcal V_i\}_{i\in\N}$ be a recursively enumerated basis for the
topology on $C(B,E)^\Z\times C(B,E)^\Z$. For each $i$, let $\chi_i(v)$ be
the $\mL$-formula saying that the $<_{L}$-least marked code for $v$ lies in
$\mathcal V_i$; explicitly, in the interpreted code sort this says
\[
\exists a\left(
 a\in \mathcal V_i\wedge a\in\mathcal C(v)\wedge
 \forall b\,(b<_{L}a\rightarrow b\notin\mathcal C(v))
\right).
\]
Here $<_{L}$, the basis $\{\mathcal V_i\}$, and the relation
$a\in\mathcal C(v)$ are all interpreted without adding symbols to the group
language; after translation through the uniform interpretation, each $\chi_i$
is an ordinary one-variable $\mL$-formula.

Since $\tp(g)=\tp(h)$, we have
\[
G_M\models \chi_i(g)\leftrightarrow \chi_i(h)
\]
for every $i$. The basis $\{\mathcal V_i\}_{i\in\N}$ separates points, so
$c(g)=c(h)$. Hence $\mathcal C(g)\cap\mathcal C(h)\neq\varnothing$, and
Proposition~\ref{prop:marked-codes-conjugacy} implies that $g$ and $h$ are
conjugate.
\end{proof}


\section{Infinitary logic and rigidity}
From now on, we will no longer assume that {V=L}.
However, the arguments needed to establish
Theorem~\ref{thm:infinitary-homeo}
and Theorem~\ref{thm:infinitary-type} are very similar to those
used to establish rigidity under {V=L}. Constructibility allows
us to pick out a canonical code for a manifold or for a manifold
marked with a homeomorphism, whereas infinitary logics allow us to completely describe arbitrary
codes with a single sentence. For the
remainder of this section, we retain the setup from Subsections
~\ref{ss:vl-fo-rigid} and~\ref{ss:vl-type}.

\begin{proof}[Proof of Theorem~\ref{thm:infinitary-homeo}]
    Suppose $M$ has dimension $d$.
    Fix a code \[(f_n)\in\mathcal M(\tau,\beta)\subseteq
    \mathcal M(M),\] and let $\chi_i(x)$ be the
    $\mathcal L$--formula expressing that an element $(h_n)\in
    C(B,E)^\Z$ lies in the basis element $\mathcal U_i$. Let
    $I=\{i\mid (f_n)\in \mathcal U_i\}$. We set $\phi(x)$ to be the
    conjunction
    \[\phi\colon \quad \bigwedge_{i\in I} \chi_i(x)\wedge \bigwedge_{j\notin I} \neg\chi_j(x).\] Then, if $N$ is an arbitrary connected manifold,
    we have 
    $$
    G_N\models \dim_d\;\;\wedge\;\; \exists\tau\;\exists\beta\;\exists
    (g_n)\;\Big((g_n)\in\mathcal M(\tau,\beta)\;\wedge\; \phi\big((g_n)\big)\Big)
    $$
    if and only if $N$ has
    dimension $d$ and $N$ admits $(f_n)$ as a code, in which case
    $M\cong N$.
\end{proof}

\begin{proof}[Proof of Theorem~\ref{thm:infinitary-type}]
    Let $g\in G_M$, and fix a marked code
    \[a=((f_n),(u_n))\in \mathcal C(\tau,\beta,g)\subseteq \mathcal C(g).
    \]
    Let $\{\mathcal V_i\}_{i\in\N}$ be a recursive basis for the topology on
    $C(B,E)^\Z\times C(B,E)^\Z$, and let $\chi_i(z)$ be the $\mL$--formula
    expressing, in the interpreted code sort, that a marked code $z$ lies in
    $\mathcal V_i$. Let
    \[I=\{i\in\N\mid a\in\mathcal V_i\}.\]
    The address of $a$ is described by the $L_{\omega_1\omega}$ conjunction
    \[
    \Phi(z)\colon\quad
    \bigwedge_{i\in I}\chi_i(z)\wedge
    \bigwedge_{j\notin I}\neg\chi_j(z).
    \]
    Now define $\psi_g(y)$ to be the $L_{\omega_1\omega}$ formula
    \[
    \exists\tau\;\exists\beta\;\exists z\;
    \big(z\in\mathcal C(\tau,\beta,y)\wedge \Phi(z)\big).
    \]
    If $h\in G_M$ and $G_M\models\psi_g(h)$, then $a$ is a marked code for
    $h$ as well as for $g$, so
    $\mathcal C(g)\cap\mathcal C(h)\neq\varnothing$. By
    Proposition~\ref{prop:marked-codes-conjugacy}, $g$ and $h$ are conjugate.
    Conversely, if $h$ is conjugate to $g$, then
    Proposition~\ref{prop:marked-codes-conjugacy} gives
    $\mathcal C(g)=\mathcal C(h)$, so the same marked code $a$ witnesses
    $G_M\models\psi_g(h)$.
\end{proof}


\section{The Baire property and non-classification}
In this section, we will show that under $\mathsf{PBP}$ there are pairs of
connected, noncompact manifolds with empty boundary whose homeomorphism groups
are elementarily equivalent but which are not homeomorphic to each other.
Similarly, we will show that two homeomorphisms of the same manifold can have
the same type and not be conjugate; this will establish
Theorems~\ref{thm:pd-notrigid} and~\ref{thm:pd-not-type-rigid}.

\subsection{Non-homeomorphic surfaces with elementarily equivalent homeomorphism groups}\label{ss:surfac-ee}
Recall that, if $X$ is a compact metrizable space, $K(X)$ denotes the hyperspace of all non-empty closed subsets of $X$ equipped with the  Vietoris topology.

In the following, let $S^2$ denote the $2$-sphere and fix a homeomorphic embedding 
$$
2^\mathbb N\overset\iota\longrightarrow S^2
$$
of the Cantor space $2^\mathbb N=\{0,1\}^\mathbb N$ into $S^2$. Let also $\mathcal C=\iota[2^\N]$ denote its image, whereby $S^2\setminus \mathcal C$ is the sphere with a Cantor set of punctures, also known as the {\em Cantor tree surface}.
We will decorate $S^2\setminus \mathcal C$ by appropriately attaching handles in such a way that under $\mathsf{PBP}$ two such decorated surfaces will end up having elementarily equivalent homeomorphism groups, but nevertheless will be non-homeomorphic. We now proceed to describe this decoration.

For all finite binary strings $t\in 2^{<\mathbb N}$, fix a closed disk $D_t\subseteq S^2 \setminus \mathcal C$ with boundary curve $C_t$ so that
\begin{enumerate}
    \item $D_t\cap \overline {\bigcup_{s\neq t}D_s}=\emptyset$, 
    \item $\overline {\bigcup_{t\in 2^{<\mathbb N}}D_t}=\mathcal C\cup \bigcup_{t\in 2^{<\mathbb N}}D_t$,
    \item for all $\alpha\in 2^\mathbb N$,  the sequence of closed sets $\big(D_{\alpha|n}\big)_{n=1}^\infty$ converges to $\{\iota(\alpha)\}$ in the Vietoris topology on $K(S^2)$.
\end{enumerate}

We let $M$ be the compact metrizable space with $S^2\subseteq M$ obtained by attaching a handle $H_t$ to $S^2$ along each boundary curve $C_t$ in such a way that
\begin{enumerate}
    \item $H_t\cap \overline {\bigcup_{s\neq t}H_s}=\emptyset$, 
    \item $\overline {\bigcup_{t\in 2^{<\mathbb N}}H_t}=\mathcal C\cup \bigcup_{t\in 2^{<\mathbb N}}H_t$,
    \item for all $\alpha\in 2^\mathbb N$,  the sequence of closed sets $\big(H_{\alpha|n}\big)_{n=1}^\infty$ converges to $\{\iota(\alpha)\}$ in the Vietoris topology on $K(M)$.
\end{enumerate}
Thus $D_t\cap H_t=C_t$  for all $t\in 2^{<\N}$.

We let 
$$
P=S^2\setminus \overline {\bigcup_{t\in 2^{<\mathbb N}}D_t} =M\setminus \overline {\bigcup_{t\in 2^{<\mathbb N}}\big(D_t\cup H_t\big)}
$$
and note that $P$ is an open subset of $S^2$ and $M$.

For $t\in 2^{<\mathbb N}$, let 
$$
N_t=\{\alpha\in 2^\mathbb N\colon t \text{ is an initial segment of } \alpha\},
$$
which is a clopen subset of $2^\mathbb N$. Thus, for every $t\in 2^{<\mathbb N}$, the set
$$
\{F\in K(2^\mathbb N)\colon F\cap N_t\neq \emptyset\}
$$
is clopen. Also, for $F\in K(2^\mathbb N)$, we let 
$$
S_F=P\;\cup\;  \bigcup_{F\cap N_t\neq \emptyset}H_t\;\cup\; \bigcup_{F\cap N_t= \emptyset}D_t
$$
and note that $S_F$ is a  decorated Cantor tree surface where an end $\iota(\alpha)$ with $\alpha\in 2^\mathbb N$ is accumulated by genus if and only if $\alpha\in F$. By the homeomorphic classification of surfaces \cite{Kerekjarto, Richards-surfaces}, it follows that two such surfaces $S_F$ and $S_{F'}$ are homeomorphic if and only if there exists a homeomorphism $h$ of the full Cantor space $2^\N$ so that  $h[F]=F'$.

Note also that, for all closed subsets $F\subseteq 2^\N$,  $S_F \subseteq M$ and that $\overline{S_F}$ is the disjoint union of $S_F$ and $\mathcal C$. We thus see that, for any two closed sets $F,F'\subseteq 2^\N$, the following two properties are equivalent.
\begin{enumerate}
    \item There is a homeomorphism $h$ of Cantor space $2^\N$ so that  $h[F]=F'$,
    \item there is a homeomorphism $h$ of $\overline{S_F}$ with $\overline{S_F'}$ so that $h[\mathcal C]=\mathcal C$.
\end{enumerate}

Let $Q=[0,1]^\N$ denote the Hilbert cube and recall the notion of $Z$-sets in $Q$. First of all, the collection of $Z$-sets  forms an ideal of closed subsets of $Q$ with the property that any homeomorphism $L\overset h \longrightarrow L'$ between two $Z$-sets extends to a homeomorphism $Q\overset{\tilde h} \longrightarrow Q$. Furthermore, $Q$ may be homeomorphically embedded into itself as a $Z$-set. As these are the only facts needed for our construction, we need not worry about the exact definition of $Z$-sets and instead refer the reader to 
\cite[Chapter 5]{vanMill} for further details.

Because every compact metric space embeds into $Q$, we can, by composing with the embedding of $Q$ into $Q$ as a $Z$-set, suppose that $M$ is itself a $Z$-subset of $Q$. In particular, this means that, for any two closed subsets \[L, L'\subseteq M\subseteq Q\] and every homeomorphism $L\overset{ h} \longrightarrow L'$, there is a homeomorphism $\tilde h$ of $Q$ extending $h$. It follows that (1) and (2) above are equivalent with
\begin{enumerate}
    \item [(3)] there is a homeomorphism $h$ of $Q$ with $h[\overline{S_F}]=\overline{S_F'}$ and  $h[\mathcal C]=\mathcal C$.
\end{enumerate}

Observe that the set
$$
\mathbb X=\Mgd{(L,h)\in K(Q)\times \Homeo(Q) }{h[L]=L \;\&\; h[\mathcal C]=\mathcal C }
$$
is closed and the map
$$
F\in K(2^\mathbb N)\mapsto \overline{S_F}\in K(Q)
$$
is continuous. It  follows from this that the set
$$
\mathbb H=\MGD{(F,h)\in K(2^\mathbb N)\times \Homeo(Q) }
{ h[\overline{S_F}]=\overline{S_F}\;\&\; h[\mathcal C]=\mathcal C }
$$
is closed.

\begin{prop}\label{prop:proj}
Suppose that $\phi$ is a first-order sentence in the language of group theory. 
Then
$$
\Mgd{F\in K(2^\mathbb N)}
{ \Homeo(S_F)\models \phi}
$$
is a projective set.
\end{prop}

\begin{proof}
By rewriting $\phi$ in prenex form with matrix in conjunctive normal form, we see that $\phi$ is equivalent to a sentence of the form
$$
\exists x_1\; \forall x_2 \;\cdots\; \exists x_{n-1}\; \forall x_n\; \bigwedge_{i=1}^p\bigvee_{j=1}^q \big(w_{ij}=u_{ij}\big)^{\epsilon_{ij}},
$$
where $w_{ij}(\overline x)$ and $u_{ij}(\overline x)$ are two words in the language of group theory and $\epsilon_{ij}\in\{-1,1\}$. Here $(w=u)^1$ designates the formula $w=u$ itself, whereas $(w=u)^{-1}$ designates its negation $w\neq u$.

Observe that, for every $F\in K(2^\mathbb N)$, the vertical section
$$
\mathbb H_F=\Mgd{h\in \Homeo(Q) }
{ h[\overline{S_F}]=\overline{S_F}\;\&\; h[\mathcal C]=\mathcal C } 
$$
is a closed subgroup of $\Homeo(Q)$.
Every homeomorphism of $S_F$ extends to the Freudenthal end compactification, and
hence to a homeomorphism of $\overline{S_F} = S_F \sqcup \mathcal C$, preserving $\mathcal C$.
Thus, the restriction to $S_F$ defines a continuous epimorphism 
$$
\mathbb H_F\overset{\pi_F}\longrightarrow \Homeo(S_F).
$$
Therefore, the statement $\phi$ holds in $\Homeo(S_F)$ if and only if 
\[\begin{split}
    &\exists f_1\in \mathbb H_F\; \forall f_2\in \mathbb H_F \;\cdots\; \exists f_{n-1}\in \mathbb H_F\; \forall f_n\in \mathbb H_F\\
&\bigwedge_{i=1}^p\bigvee_{j=1}^q \Big(w_{ij}\big(\pi_F(f_1),\ldots,\pi_F(f_n)\big)=u_{ij}\big(\pi_F(f_1),\ldots,\pi_F(f_n)\big)\Big)^{\epsilon_{ij}}.
\end{split}\]
Suppose now that $w(x_1,\ldots, x_n)$ and $u(x_1,\ldots, x_n)$ are two words in the language of group theory. We let $  [w=u]$ denote the set of all tuples 
$$
(f_1,\ldots, f_n, F)\in  K(2^\mathbb N)\times \Homeo(Q)^n 
$$
satisfying
$$
w\big(\pi_F(f_1),\ldots,\pi_F(f_n)\big)=u\big(\pi_F(f_1),\ldots,\pi_F(f_n)\big),
$$
that is, so that
$$
\forall z\in Q\;\Big( z\in S_F\;\to\; w\big(f_1,\ldots,f_n\big)(z)=u\big(f_1,\ldots,f_n\big)(z)\Big).
$$
By the latter expression, we find that $[w=u]$ is a coanalytic subset of $K(2^\mathbb N)\times \Homeo(Q)^n$. Expanding the quantifiers $\exists f\in \mathbb H_F$ and $\forall f\in \mathbb H_F$ as 
$$
\exists f\; \big( (F,f)\in \mathbb H \;\;\&\;\; \ldots\;\big) \quad\text{and}\quad  \forall f\; \big( (F,f)\in \mathbb H \;\rightarrow\; \ldots\;\big),
$$
we finally see that 
\[\begin{split}
    &\exists f_1\in \mathbb H_F\; \forall f_2\in \mathbb H_F \;\cdots\; \exists f_{n-1}\in \mathbb H_F\; \forall f_n\in \mathbb H_F\\
&\bigwedge_{i=1}^p\bigvee_{j=1}^q \Big(w_{ij}\big(\pi_F(f_1),\ldots,\pi_F(f_n)\big)=u_{ij}\big(\pi_F(f_1),\ldots,\pi_F(f_n)\big)\Big)^{\epsilon_{ij}}
\end{split}\]
defines a projective condition on $F$. In other words, the set
$$
\{F\in K(2^\mathbb N)\colon \Homeo(S_F)\models \phi\}
$$
is projective.
\end{proof}

\begin{thm}[Assume $\mathsf{PBP}$]
  \label{thm:proj-det-non-class}
There are non-homeomorphic orientable surfaces with elementarily equivalent homeomorphism groups.
\end{thm}

\begin{proof}
Let $E_0$ denote the equivalence relation of eventual agreement of infinite binary sequences, that is, for $\alpha,\beta\in 2^\mathbb N$, we set
$$
\alpha E_0\beta \; \Leftrightarrow \; \exists m\; \forall n\geqslant m\; \;\alpha_n=\beta_n.
$$
By \cite[Theorem 3]{camerlo}, there is a Borel measurable map
$$
2^\mathbb  N\overset \kappa\longrightarrow K(2^\mathbb N)
$$
so that 
$$
\alpha E_0\beta \; \Leftrightarrow \exists h\in \Homeo(2^\mathbb N)\;\; h[\kappa(\alpha)]=\kappa(\beta).
$$
By Proposition \ref{prop:proj}, for every first-order sentence of the language of group theory $\phi$, the set 
$$
\Mgd{F\in K(2^\mathbb  N)}{\Homeo(S_F)\models \phi}
$$
is projective. As the inverse image of a projective set by a Borel function is also projective, it follows that
$$
A_\phi=\Mgd{\alpha\in 2^\mathbb N}{\Homeo(S_{\kappa(\alpha)})\models \phi}
$$
is projective too. Moreover, for all $\alpha,\beta\in 2^\mathbb N$, 
\[\begin{split}
\alpha E_0\beta
&\Rightarrow
\exists h\in \Homeo(2^\mathbb N)\;\; h[\kappa(\alpha)]=\kappa(\beta)\\
&\Rightarrow
S_{\kappa(\alpha)}\cong S_{\kappa(\beta)}\\
&\Rightarrow
\Homeo(S_{\kappa(\alpha)})\cong \Homeo(S_{\kappa(\beta)})\\
&\Rightarrow
(\alpha\in A_\phi \leftrightarrow \beta\in A_\phi).
\end{split}\]
In other words, $A_\phi$ is an $E_0$-invariant projective set. 

Now, under $\mathsf{PBP}$, the projective set $A_\phi$ has the Baire property. Since $A_\phi$ is $E_0$-invariant, the category zero-one law for $E_0$ implies that $A_\phi$ is either meager or comeager. If $A_\phi$ is comeager, let $C_\phi=A_\phi$ and otherwise let $C_\phi=2^\mathbb N\setminus A_\phi$. It thus follows that
$$
C=\bigcap_{\phi}C_\phi,
$$
where the intersection runs over the countable collection of all first-order sentences $\phi$ of the language of group theory, is comeager in $2^\mathbb N$. Furthermore, for all $\alpha,\beta\in C$, we have
$$
\Homeo(S_{\kappa(\alpha)})\equiv \Homeo(S_{\kappa(\beta)}).
$$
Since every $E_0$-class is countable and $C$ is comeager, it suffices to pick $\alpha,\beta\in C$ that are $E_0$-inequivalent and hence so that $S_{\kappa(\alpha)}\not\cong S_{\kappa(\beta)}$, but nevertheless $\Homeo(S_{\kappa(\alpha)})\equiv \Homeo(S_{\kappa(\beta)})$.
\end{proof}


\subsection{Nonhomeomorphic manifolds with elementarily equivalent homeomorphism groups}

We now wish to show that the failure of first-order rigidity for
homeomorphism groups of noncompact surfaces implies failure of first-order
rigidity for noncompact manifolds in all higher dimensions. For an arbitrary
connected surface $S$, we will write $P_k(S)$ for $S\times S^k$, the
Cartesian product of $S$ with the $k$--sphere $S^k$. we will suppress the
notation $S$ and $k$ when they are clear from context or irrelevant.

The goal of this subsection is to show the following.
\begin{thm}
  For every fixed $k\geqslant 1$ the manifold $P_{k}(S)$ and its homeomorphism group $\Homeo(P_{k}(S))$ are uniformly interpretable in $\Homeo(S)$.
\end{thm}

Observe that there is a natural map $\pi\colon P_k(S)\longrightarrow S$ by
projecting onto the first factor. Moreover, if $K\subseteq S$ is a compact
subspace then $\pi^{-1}(K)$ is a compact subspace of $P_k(S)$. It is
straightforward then that the  end-space $\mathcal E(S)$ is homeomorphic to the end-space $\mathcal E(P_k(S))$.

Let \[K_0\subseteq K_1\subseteq\cdots\] be a
cofinal sequence of increasing compact subspaces
of $S$ and let $U_i=S\setminus K_i$. We let $(\hat U_i)$ denote a
coherent choice of connected component of each $U_i$, so that
$\hat U_i$ is a connected component of $U_i$ for each $i$ and 
$\hat U_{i+1}\subseteq \hat U_i$ for each $i$. Thus, $(\hat U_i)$
defines an end of the surface $S$. It follows then that
$(\hat V_i)=(\hat U_i\times S^k)$, given by taking the Cartesian
product of each $\hat U_i$ with $S^k$, defines an end of $P_k(S)$.
We will say that $(\hat V_i)$ is \emph{accumulated by genus}
if $(\hat U_i)$ is. The set of ends accumulated by genus
is written $\mathcal E^g(P_k(S))$.

\begin{lem}\label{lem:product}
    The inclusion $\mathcal E^g(P_k(S))\subseteq \mathcal E(P_k(S))$ is
    $\Homeo(P_k(S))$--invariant.
\end{lem}
\begin{proof}
    It suffices to distinguish topologically between the ends of
    $P_k(S)$ which are accumulated by genus and those which are not, since
    a topological distinction will be preserved by all homeomorphisms.
    We observe that an end $(\hat V_i)$ is accumulated by
    genus if and only if there exists a disjoint collection
    of compact submanifolds $\{W_i\}_{i\in\N}$ of $P_k(S)$
    with $W_i\subseteq\hat V_i$ for all $i$ and
    such that:
    \begin{enumerate}
        \item For all $i$, the manifold $W_i$ admits a locally
        separating submanifold
        \[M_i=S^1\times S^k\times\{0\}\subseteq W_i\] such that
        $W_i\cong M_i\times [-1,1]$;
        \item For all $i$, the complement $N_i=P_k(S)\setminus W_i$
        is connected.
    \end{enumerate}

    We claim that such a sequence $\{W_i\}_{i\in\N}$ exists if and only
    if for all $i$, the component $\hat U_i\subseteq S$ contains a
    simple closed curve which is nonseparating in $S$; this latter
    property is possible if and only if the end 
    $(\hat V_i)$ is accumulated by genus.

    Note that if $(\hat V_i)$ is accumulated by genus
    then such a collection $\{W_i\}_{i\in\N}$ clearly exists.
    Indeed, in each $\hat U_i$ there is a nonseparating essential
    closed curve
    $\gamma_i$, and we may simply take $W_i$ to be the closure
    of a small tubular
    neighborhood of $\gamma_i$ crossed with $S^k$. Conversely,
    if $(\hat V_i)$ is not accumulated by genus then
    no such construction is possible.

    In general, we may assume $M_i$ locally separates $P_k(S)$ but
    does not globally separate $P_k(S)$. Choose a loop $\delta$ in
    $P_k(S)$ which intersects $M_i$ exactly once. There is a compact
    essential
    subsurface $\Sigma\subseteq S$ such that $\delta\cup M_i\subseteq
    \Sigma\times S^k$, and we claim that $\Sigma$ must have positive
    genus, and suppose the contrary for a contradiction.
    
    We view $M_i$ as an element of
    $H_{k+1}(\Sigma\times S^k,\bZ/2\bZ)$. Writing $\partial$ for the
    boundary of $\Sigma\times S^k$,
    Poincar\'e--Lefschetz duality implies that intersection number with
    elements of $H_{k+1}(\Sigma\times S^k,\Z/2\Z)$ furnishes an isomorphism with
    $H^1(\Sigma\times S^k,\partial,\Z/2\Z)$. If $\Sigma$ has genus $0$
    then the long exact sequence on relative cohomology shows that
    the inclusion of the boundary $\partial$ of $\Sigma\times S^k$
    induces a surjection \[H^0(\partial,\Z/2\Z)\longrightarrow
    H^1(\Sigma\times S^k,\partial,\Z/2\Z).\] On the other hand,
    $M_i$ furnishes a nontrivial
    $\Z/2\Z$--valued first cohomology class of $\Sigma\times S^k$
    relative to the boundary, which is not in the image of
    $H^0(\partial,\Z/2\Z)$. This is a contradiction.
\end{proof}

\begin{cor}
    Assume that $\mathsf{PBP}$ holds and let $d\geqslant 2$. There exist pairs of non-homeo\-morphic
    $d$-manifolds whose homeomorphism groups are elementarily equivalent.
\end{cor}
\begin{proof}
    By Theorem~\ref{thm:pd-dim-implication}
    and Lemma~\ref{lem:product},
    it suffices to find two surfaces which are not homeomorphic but
    have elementarily equivalent homeomorphism groups, and which
    remain non-homeomorphic after taking a Cartesian product
    with spheres of positive dimension. This
    follows immediately from the construction of the surfaces
    $S_F$ in 
    Section
    ~\ref{ss:surfac-ee} above and Lemma~\ref{lem:product}.
\end{proof}


\subsection{Failure of type rigidity under the Baire-property hypothesis}
In the present section, we let $M$ be a fixed manifold of dimension $d\geqslant 1$. For this section, we will not assume that
$M$ has no boundary nor that it be connected.

Let $D_\infty$ denote the infinite dihedral group viewed as the group of order-preserving and order-reversing automorphisms of the linear order $(\Z,<)$. We define an action $D_\infty\curvearrowright \{-1,1\}^\Z$ by letting
$$
(\sigma\cdot \alpha)_{\sigma(i)}=\begin{cases}
    \alpha_i&\text{if $\sigma$ is order-preserving},\\
-\alpha_i&\text{if $\sigma$ is order-reversing},
\end{cases}
$$
for all $\sigma\in D_\infty$, $\alpha\in \{-1,1\}^\Z$ and $i\in \Z$. Observe also that by restricting to the subgroup $\Z\leqslant D_\infty$ of translations, we obtain the usual shift-action $\Z\curvearrowright \{-1,1\}^\Z$.
We will use the corresponding category zero-one law for the shift: every shift-invariant subset of $\{-1,1\}^\Z$ with the Baire property is either meager or comeager.

\begin{prop}\label{prop:redconj}
There is a continuous map $\{-1,1\}^\mathbb Z\overset f\longrightarrow \Homeo_0([0,1])$ such that
\[\begin{split}
    \exists \sigma \in \Z \;\big(\sigma\!\cdot\! \alpha=\beta\big)
    \quad\Leftrightarrow\quad
    \exists h\in \Homeo_0([0,1])\;\big( hf_\alpha h^{-1}=f_\beta\big)
\end{split}\]   
and
\[\begin{split}
    \exists \sigma \in D_\infty \;\big(\sigma\!\cdot\! \alpha=\beta\big)
    \quad\Leftrightarrow\quad
    \exists h\in \Homeo([0,1])\;\big( hf_\alpha h^{-1}=f_\beta\big).
\end{split}\]   
\end{prop}

\begin{proof}
Let $i\in \Z\mapsto p_i\in (0,1)$ be an order-embedding of $\Z$ into the open interval $(0,1)$ so that $\lim_{i\to \infty}p_{-i}=0$ and $\lim_{i\to \infty}p_i=1$. Fix also a homeomorphism $\zeta_i$ of $J_i=(p_i,p_{i+1})$ such that $x<\zeta_i(x)$  for all $x\in J_i$ and therefore also $\zeta_i^{-1}(x)<x$ for all $x\in J_i$.

For each $\alpha\in \{-1,1\}^{\Z}$, we construct 
$$
f_{\alpha}\in \Homeo_0([0,1])
$$ 
by letting $f_\alpha(0)=0$, $f_\alpha(1)=1$ and 
$$
f_{\alpha}\!\upharpoonright_{J_i}=\zeta_i^{\alpha_i}\qquad\&\qquad f_\alpha(p_i)=p_i
$$
for all $i\in \Z$. The map $\alpha\mapsto f_{\alpha}$ is evidently continuous and $\{0,1\}\cup\mgd{p_i}{i\in\Z}$ is exactly the set of fixed points of $f_\alpha$.

Suppose $f_\beta h=hf_\alpha$ for some $h\in \Homeo([0,1])$ and $\alpha,\beta\in \{-1,1\}^\Z$. Then $h$ maps the fixed points of $f_\alpha$ to those of $f_\beta$ and so there is some $\sigma\in D_\infty$ such that $h[J_i]=J_{\sigma(i)}$ for all $i\in \Z$. In particular, $\sigma$ is order-preserving if and only if $h$ is orientation-preserving.
Note that, if $h$ is orientation-preserving, then for all $i\in \Z$
\[\begin{split}
\alpha_i=1 
\;\Leftrightarrow\;& \forall x \in J_i\quad x<f_\alpha(x)\\
\;\Leftrightarrow\;& \forall x \in J_i\quad h(x)<hf_\alpha(x)=f_\beta h(x)\\
\;\Leftrightarrow\;& \forall y \in J_{\sigma(i)}\quad y<f_\beta (y)\\
\;\Leftrightarrow\;& \beta_{\sigma(i)}=1,
\end{split}\]
whereas, if $h$ is orientation-reversing,
\[\begin{split}
\alpha_i=1 
\;\Leftrightarrow\;& \forall x \in J_i\; x<f_\alpha(x)\\
\;\Leftrightarrow\;& \forall x \in J_i\; f_\beta h(x)=hf_\alpha(x)<h(x)\\
\;\Leftrightarrow\;& \forall y \in J_{\sigma(i)}\; f_\beta (y)<y\\
\;\Leftrightarrow\;& \beta_{\sigma(i)}=-1.
\end{split}\]
Thus, in either case, we find that $\beta=\sigma\cdot\alpha$.

For the converse, note that any two $\zeta_i$ and $\zeta_j$ are conjugate by an orientation-preserving homeomorphism $J_i\longrightarrow J_j$, whereas $\zeta_i$ and $\zeta_j^{-1}$ are conjugate by an orientation-reversing homeomorphism $J_i\longrightarrow J_j$. Using this, one easily sees that, when $\beta=\sigma\cdot\alpha$ for some $\sigma\in D_\infty$, then also $f_\beta h=hf_\alpha$ for some $h\in \Homeo([0,1])$.
    
This shows that
\[\begin{split}
    \exists \sigma \in D_\infty \;\big(\sigma\!\cdot\! \alpha=\beta\big)
    \quad\Leftrightarrow\quad
    \exists h\in \Homeo([0,1])\;\big( hf_\alpha h^{-1}=f_\beta\big).
\end{split}\]   
To get the first equivalence, we simply note that $\sigma$ is order-preserving if and only if $\sigma\in \Z$, whereas $h$ is order-preserving if and only if $h\in \Homeo_0([0,1])$.
\end{proof}

\begin{prop}\label{prop:redconj2}
For every $d\geqslant 2$ there is a continuous map $\{-1,1\}^\mathbb Z\overset g\longrightarrow \Homeo_\partial\big(B^d(1)\big)$ such that
\[\begin{split}
    \exists \sigma \in D_\infty \;\big(\sigma\!\cdot\! \alpha=\beta\big)
    &\quad\Leftrightarrow\quad
    \exists h\in \Homeo\big(B^d(1)\big)\;\big( hg_\alpha h^{-1}=g_\beta\big)\\
    &\quad\Leftrightarrow\quad
    \exists h\in \Homeo_0\big(B^d(1)\big)\;\big( hg_\alpha h^{-1}=g_\beta\big).
\end{split}\]   
\end{prop}

\begin{proof}
Recall that, for a topological space $X$, $\Sigma X$ denotes the suspension over $X$, that is, the product $X\times [0,1]$ with $X\times \{1\} $ and $X\times \{0\}$ collapsed to single points $p^+$ and $p^-$ respectively. Observe that every homeomorphism of $X$ extends canonically to a homeomorphism of the suspension $\Sigma X$.

 We let $\Sigma^{d-1}[0,1]$ denote the $d-1$-fold suspension
    $$
    \Sigma^{d-1}[0,1]= \underbrace{\Sigma(\cdots \Sigma(\Sigma}_{d-1 \text{ times }} [0,1])\cdots)\cong B^d(1).
    $$
So, by induction, the $f_\alpha\in \Homeo_\partial([0,1])$ constructed in the proof of Proposition \ref{prop:redconj} extend to homeomorphisms $g_\alpha\in \Homeo_\partial\big(\Sigma^{d-1}[0,1]\big)$ so that
$$
\overline{\mgd{x\in \Sigma^{d-1}[0,1]}{g_\alpha(x)\neq x}}=\Sigma^{d-1}[0,1].
$$ 

We also note that there is an orientation-reversing $k\in \Homeo\big(\Sigma^{d-1}[0,1]\big)$ such that $g_\alpha k=kg_\alpha$ for all $\alpha$. Indeed, as $d\geqslant 2$, we may let $\kappa$ be the orientation-reversing homeomorphism of $\Sigma [0,1]$ defined by $\kappa(r,t)=(r,1-t)$ and then simply let $k$ be the induced homeomorphism of $\Sigma^{d-1}[0,1]$. It thus follows that if $h\in \Homeo\big(\Sigma^{d-1}[0,1]\big)$ is orientation-reversing and $g_\beta h=hg_\alpha$, then also $g_\beta kh=khg_\alpha$ and $kh\in \Homeo_0\big(\Sigma^{d-1}[0,1]\big)$. This verifies the last equivalence of the proposition.

Suppose now that $\sigma\cdot \alpha=\beta$ for some $\sigma\in D_\infty$ and $\alpha,\beta\in \{-1,1\}^\Z$. Then $f_\beta h=hf_\alpha$ for some $h\in \Homeo([0,1])$ and, if $\tilde h$ denotes the induced homeomorphism $\tilde h\in \Homeo\big(\Sigma^{d-1}[0,1]\big)$, then $g_\beta \tilde h=\tilde hg_{\alpha}$.

For the converse, note first that as $\Sigma^{d-1}[0,1]\cong B^d(1)$ we may talk about its interior, which can be seen to be a union of $\Z$--indexed regions $\{U_i\}_{i\in\Z}$, each of which is invariant under all $g_{\alpha}$. The interior of each $U_i$ is itself homeo\-morphic to $\R^d$, whereas two regions $U_{i-1}$ and $U_{i}$ meet along a copy $P_{i}$ of $\R^{d-1}$ that is pointwise fixed by all $g_\alpha$. Furthermore, 
 by appropriately choosing the homeomorphisms ${\sf int} (U_i)\cong \R^{d}$,  up to conjugacy, all $g_{\alpha}$ act on each such copy of $\R^d$ by translation either to the left or to the right with respect to the coordinate axis $x_1$ according to whether  $\alpha(i)={-1}$ or $\alpha(i)=1$. In particular, $\bigcup_{i\in \Z}P_i$ is exactly the collection of points in the interior of $\Sigma^{d-1}[0,1]$ fixed by the $g_\alpha$.

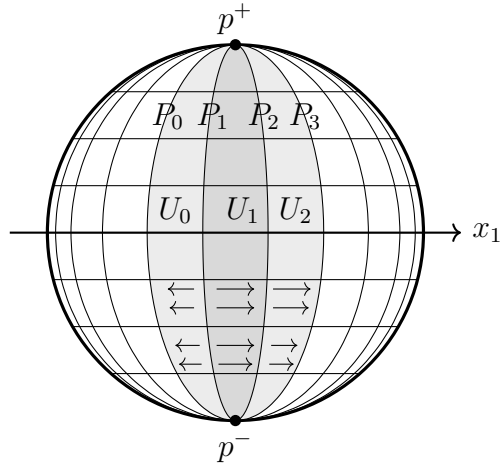
\begin{center}
 \begin{tikzpicture}[scale=2.5]

\draw[line width=1.2pt] (0,0) circle (1);

\begin{scope}
    \fill[gray!30]
        plot[domain=-90:90,variable=\phi] 
            ({cos(\phi)*sin(-10)},  {sin(\phi)})
        --
        plot[domain=90:-90,variable=\phi] 
            ({cos(\phi)*sin(10)}, {sin(\phi)})
        -- cycle;

    \fill[gray!15]
        plot[domain=-90:90,variable=\phi] 
            ({cos(\phi)*sin(10)},  {sin(\phi)})
        --
        plot[domain=90:-90,variable=\phi] 
            ({cos(\phi)*sin(28)}, {sin(\phi)})
        -- cycle;

    \fill[gray!15]
        plot[domain=-90:90,variable=\phi] 
            ({cos(\phi)*sin(-28)},  {sin(\phi)})
        --
        plot[domain=90:-90,variable=\phi] 
            ({cos(\phi)*sin(-10)}, {sin(\phi)})
        -- cycle;
\end{scope}

\foreach \lambda in {-73,-61,-45,-28,-10,10,28,45,61,73}{
    \draw[smooth,domain=-90:90,variable=\phi]
        plot ({cos(\phi)*sin(\lambda)}, {sin(\phi)});
}

\foreach \y in {-0.75,-0.50,-0.25,0,0.25,0.50,0.75}{
    \draw (-{sqrt(1-\y*\y)},\y) -- ({sqrt(1-\y*\y)},\y);
}

\fill (0,1) circle (0.03);
\fill (0,-1) circle (0.03);
\node[above] at (0,1) {$p^+$};
\node[below] at (0,-1) {$p^-$};

\node at (0.04,0.12) {$U_1$};
\node at (0.32,0.12) {$U_2$};
\node at (-0.32,0.12) {$U_0$};

\node at (-0.36,0.62) {$P_0$};
\node at (-0.12,0.62) {$P_1$};
\node at (0.15,0.62) {$P_2$};
\node at (0.37,0.62) {$P_3$};

\draw[->, thick] (-1.2,0) -- (1.2,0) node[right] {$x_1$};

\draw[->] (-.1,-.3) -- (.1,-.3);
\draw[->] (-.1,-.4) -- (.1,-.4);
\draw[->] (-.1,-.6) -- (.1,-.6);
\draw[->] (-.09,-.7) -- (.09,-.7);

\draw[<-] (-.36,-.3) -- (-.22,-.3);
\draw[<-] (-.35,-.4) -- (-.22,-.4);
\draw[<-] (-.32,-.6) -- (-.19,-.6);
\draw[<-] (-.3,-.7) -- (-.18,-.7);

\draw[->] (.2,-.3) -- (.4,-.3);
\draw[->] (.2,-.4) -- (.39,-.4);
\draw[->] (.19,-.6) -- (.33,-.6);
\draw[->] (.18,-.7) -- (.31,-.7);

\end{tikzpicture}

\end{center}
 
Observe that, for all $\alpha\in\{-1,1\}^\Z$,  $i\in \Z$, and $x\in {\sf int}(U_i)$,
\[
\begin{split}
    \alpha_i=-1 &\quad\Leftrightarrow\quad \lim_{n\to \infty}g_{\alpha}^n(x)\in P_{i}\\
    \alpha_i=1 &\quad\Leftrightarrow\quad \lim_{n\to \infty}g_{\alpha}^n(x)\in P_{i+1}.
\end{split}
\]

Suppose now that $g_\beta h=hg_\alpha$ for some $h\in \Homeo\big(\Sigma^{d-1}[0,1]\big)$. Then $h$ preserves the interior of $\Sigma^{d-1}[0,1]$ and must therefore map $\bigcup_{i\in \Z}P_i$ to itself. It follows that there is some $\sigma\in D_\infty$ so that $h[U_i]=U_{\sigma(i)}$ for all $i\in \Z$. On the other hand, since $P_{i}=U_{i-1}\cap U_{i}$, we have
$$
h[P_i]=h[U_{i-1}]\cap h[U_{i}]=U_{\sigma(i-1)}\cap U_{\sigma(i)}=
\begin{cases}
P_{\sigma(i)}&\text{if $\sigma$ is order-preserving},\\
P_{\sigma(i)+1}&\text{if $\sigma$ is order-reversing}.
\end{cases}
$$

Suppose $\sigma$ is order-preserving. Then, for all $i$
\[\begin{split}
\alpha_i=-1 
\;\Leftrightarrow\;& \forall x \in {\sf int}(U_i)\; \lim_{n\to \infty}g_{\alpha}^n(x)\in P_{i}\\
\;\Leftrightarrow\;& \forall x \in {\sf int}(U_i)\; \lim_{n\to \infty}g_{\beta}^nh(x)=\lim_{n\to \infty}hg_{\alpha}^n(x)=h\big(\lim_{n\to \infty}g_{\alpha}^n(x)\big)\\
&\hspace{7cm}\in h[P_{i}]=P_{\sigma(i)}\\
\;\Leftrightarrow\;& \forall y \in {\sf int}(U_{\sigma(i)})\; \lim_{n\to \infty}g_{\beta}^n(y)\in P_{\sigma(i)}\\
\;\Leftrightarrow\;& \beta_{\sigma(i)}=-1.
\end{split}\]
On the other hand, if $\sigma$ is order-reversing
\[\begin{split}
\alpha_i=-1 
\;\Leftrightarrow\;& \forall x \in {\sf int}(U_i)\; \lim_{n\to \infty}g_{\alpha}^n(x)\in P_{i}\\
\;\Leftrightarrow\;& \forall x \in {\sf int}(U_i)\; \lim_{n\to \infty}g_{\beta}^nh(x)=\lim_{n\to \infty}hg_{\alpha}^n(x)=h\big(\lim_{n\to \infty}g_{\alpha}^n(x)\big)\\
&\hspace{7cm}\in h[P_{i}]=P_{\sigma(i)+1}\\
\;\Leftrightarrow\;& \forall y \in {\sf int}(U_{\sigma(i)})\; \lim_{n\to \infty}g_{\beta}^n(y)\in P_{\sigma(i)+1}\\
\;\Leftrightarrow\;& \beta_{\sigma(i)}=1.
\end{split}\]
Since this holds for all $i\in \Z$ we find that $\beta=\sigma\cdot \alpha$.
\end{proof}

We can now prove the failure of type rigidity under $\mathsf{PBP}$.

\begin{thm}
Assume $\mathsf{PBP}$.
    For every manifold $M$ of dimension $d\geqslant 1$, there are $g_1,g_2\in \Homeo_0(M)$  that are non-conjugate in $\Homeo(M)$, whereas
    $$
    \tp(g_1)=
    \tp(g_2).
    $$
\end{thm}

\begin{proof}
Recall our notation $S^{d-1}=\partial \big(B^d(1)\big)$.
Suppose $h\in \Homeo_0\big(B^d(1)\big)$, whereby 
$$
h\upharpoonright_{S^{d-1}}\in \Homeo_0\big(S^{d-1}\big).
$$
Since $\Homeo_0(S^{d-1})$ is path connected, we may find a continuous path 
$$
(h_t)_{t\in [1,2]}\in \Homeo_0\big(S^{d-1}\big)
$$
beginning at $h_1=h\upharpoonright_{S^{d-1}}$ and so that $h_2={\sf id}_{S^{d-1}}$. We may therefore define a homeomorphism $\tilde h\in \Homeo_\partial\big(B^d(2)\big)$ that extends $h$ by letting
$$
\tilde h(x)=
\begin{cases}
h(x)&\text{if }\|x\|\leqslant 1,\\
rh_r\big(\tfrac x{r}\big)&\text{if }1<\|x\|=r\leqslant 2.
\end{cases}
$$

Fix now a homeomorphic embedding of $B^d(2)$ into $M$ and identify $B^d(2)$ with its image. Then every element of $\Homeo_\partial\big(B^d(2)\big)$ extends to an element of $\Homeo(M)$ by setting it to be the identity on $M\setminus B^d(2)$.
By the argument above, we thus see that every $h\in \Homeo_0\big(B^d(1)\big)$ canonically extends to a full homeomorphism of $M$, which is the identity on $M\setminus B^d(2)$. Similarly, every $g\in \Homeo_\partial\big(B^d(1)\big)$ extends to all of $M$ by setting it to be the identity on $M\setminus B^d(1)$.

Assume first $d\geqslant 2$. Then, by combining the above discussion with Proposition \ref{prop:redconj2}, we obtain a continuous map
$$
\{-1,1\}^\Z\overset g\longrightarrow \Homeo_0(M)
$$
such that 
\[\begin{split}
    \exists \sigma \in D_\infty \;\big(\sigma\!\cdot\! \alpha=\beta\big)
    &\quad\Leftrightarrow\quad
    \exists h\in \Homeo(M)\;\big( hg_\alpha h^{-1}=g_\beta\big).
\end{split}\]

Observe that, for every first-order formula $\phi(x)$ of the language of group theory, the set
$$
\Mgd{f\in \Homeo_0(M)}{\Homeo(M)\models \phi(f)}
$$
is projective and invariant under conjugacy by $\Homeo(M)$. The reason for this is that $\Homeo(M)$ can be coded as a closed subset of a Polish function space. Equality, multiplication, and inversion become Borel in that code. Satisfaction of a first-order formula is then projective by induction on quantifier complexity, with quantifiers moving the solution set up at most one level in the projective hierarchy.

It follows that
$$
A_\phi=\Mgd{\alpha\in \{-1,1\}^\Z}{\Homeo(M)\models \phi(g_{\alpha})}
$$
is a $D_\infty$-invariant projective set. By $\mathsf{PBP}$, it has the Baire property; since it is in particular shift-invariant, the category zero-one law implies that it is either meager or comeager. Let $C_\phi=A_\phi$ if $A_\phi$ is comeager and let $C_\phi=\{-1,1\}^\Z\setminus A_\phi$ otherwise. Then
$$
C=\bigcap_{\phi}C_\phi,
$$
where the intersection runs over the countable collection of all first-order formulae $\phi(x)$ of the language of group theory, is comeager in $\{-1,1\}^\Z$. Furthermore, for all $\alpha,\beta\in C$, we have
$$
\tp(g_\alpha)=\tp(g_\beta).
$$
Since every $D_\infty$-orbit is countable and $C$ is comeager, it suffices to pick $\alpha,\beta\in C$ that are $D_\infty$-orbit inequivalent and hence such that $g_\alpha$ and $g_\beta$ are non-conjugate in $\Homeo(M)$.

The argument for the case $d=1$ is similar and uses Proposition \ref{prop:redconj} in place of \ref{prop:redconj2}.
\end{proof}


\section{Consistency of non-classification over ZFC alone}

Shelah proved that every model of ZFC has a forcing extension satisfying $\mathsf{PBP}$~\cite{Shelah1984}, so Theorem~\ref{thm:pd-notrigid} already gives Theorem~\ref{thm:consistent-intro}. We nevertheless include a direct forcing argument, which shows explicitly that the conclusion of Theorem~\ref{thm:proj-det-non-class} has no consistency strength beyond that of ZFC. In particular, we show that the consistency of ZFC implies the consistency of ZFC + ``for every $d\geqslant 2$, there are non-homeomorphic orientable $d$--dimensional manifolds with elementarily equivalent homeomorphism groups'' by forcing to collapse $\mathfrak{c}^+$ to be $\omega$ (where here $\mathfrak{c}=|\mathbb R|$ is the cardinality of the continuum). In this section, we will use standard arguments from forcing; see~\cite{hinman-book,Krivine,weaver-book} for relevant background.

Fix a countable transitive
model $V$ of ZFC and let $\kappa=\mathfrak{c}^+$ as computed in $V$. Let also $\col(\omega,\kappa)$ be the partial order consisting of all partial functions 
$$
p\colon \omega\rightharpoonup \kappa
$$
with finite domain and in which $p\leqslant q$ when $p$ extends $q$ as a function. Thus $p\leqslant \emptyset$ where $\emptyset$ denotes the function with empty domain. Elements of $\col(\omega,\kappa)$ are called {\em forcing conditions}. For all $m<\omega$ and $\alpha<\kappa$, let 
$$
D_{m,\alpha}=\mgd{p\in \col(\omega,\kappa)}{m\in {\sf dom}(p)\;\&\; \alpha\in {\sf im}(p)}
$$
and note that $D_{m,\alpha}$ is {\em dense} in $\col(\omega,\kappa)$, meaning that each $q$ has a minorant $p\leqslant q$ in $D_{m,\alpha}$. A {\em generic extension} $V[G]$ of $V$ is the smallest transitive model of ZFC containing the submodel $V$ such that $G\in V[G]$, where $G$ is some {\em $\col(\omega,\kappa)$-generic filter}. The latter means that $G\subseteq \col(\omega,\kappa)$ is an upwards closed subset such that any two $p,q\in G$ have a common minorant $r\in G$ and furthermore such that $G$ meets every dense subset of $\col(\omega,\kappa)$. In particular, in $V[G]$, we obtain a surjection
$$
g=\bigcup G \colon \omega\longrightarrow\kappa
$$
and therefore  $|\kappa|=|\omega|=\aleph_0$ in $V[G]$. It follows that every ordinal $\alpha<\kappa$ is countable in $V[G]$.

We briefly recall a few facts about the forcing relation $\Vdash_{\mathrm{Col}(\omega,\kappa)}$. For  a parameter-free formula $\chi(x_1,\ldots,x_m)$, elements $a_1,\ldots, a_m\in V$ and a forcing condition $p\in \col(\omega,\kappa)$, we have
$$
p\Vdash_{\mathrm{Col}(\omega,\kappa)}\chi(\check{a_1},\ldots,\check{a_m}) \quad\Leftrightarrow\quad \forall G\ni p\quad V[G]\models \chi({a_1},\ldots,{a_m}).
$$
Similarly, for a fixed generic filter $G$, we have
$$
V[G]\models \chi({a_1},\ldots,{a_m}) \quad\Leftrightarrow\quad \exists q\in G\quad q\Vdash_{\mathrm{Col}(\omega,\kappa)}\chi(\check{a_1},\ldots,\check{a_m}).
$$

The partial order $\col(\omega,\kappa)$ is easily seen to be {\em homogeneous}, meaning that, for every two elements $p,q\in \col(\omega,\kappa)$, there is an automorphism $\theta$ of $\col(\omega,\kappa)$ so that $\theta(p)$ and $q$ have a common minorant. We also note that, if $\chi(x_1,\ldots,x_m)$ is parameter-free and $a_1,\ldots, a_m\in V$, then 
$$
p\Vdash_{\mathrm{Col}(\omega,\kappa)} \chi(\check{a_1},\ldots,\check{a_m}) \;\;\Leftrightarrow\;\;
\theta(p)\Vdash_{\mathrm{Col}(\omega,\kappa)} \chi(\check{a_1},\ldots,\check{a_m})
$$
\cite[pp. 155-157]{Krivine}.  It thus follows from homogeneity that
the empty condition forces every formula or its negation, i.e.~either
$$
\emptyset\Vdash_{\mathrm{Col}(\omega,\kappa)} \chi(\check{a_1},\ldots,\check{a_m}) $$
or
$$
\emptyset\Vdash_{\mathrm{Col}(\omega,\kappa)} \neg\chi(\check{a_1},\ldots,\check{a_m}).
$$

\begin{lem}\label{lem:homogeneous-function}
Suppose that $\phi(x,y)$ is a parameter-free formula that defines a function from $\omega_1$ to $2^{\omega}$. Then there is a formula  $\psi(x,y)$ with unique parameter $\kappa$ that defines a function from $\kappa$ to $2^{\omega}$ such that, for any
$\alpha < \kappa$, 
\[\emptyset \Vdash_{\mathrm{Col}(\omega,\kappa)} \phi(\check{\alpha},\check{r})\quad \textrm{if and only if}\quad  \psi(\alpha,r).\]
\end{lem}

\begin{proof}
    The relation $\emptyset \Vdash_{\mathrm{Col}(\omega,\kappa)} \phi(\check{\alpha},\check{r})$ is definable in the parameter $\kappa$ by the uniform definability of $\mathrm{Col}(\omega,\kappa)$ and the definability of forcing, so let $\psi(x,y)$ be a formula with unique parameter $\kappa$ defining this relation. It is immediate that for any  $\alpha < \kappa$ there is at most one $r \in 2^{\omega}$ such that $\psi(\alpha,r)$, so all we need to verify is that such an $r$ exists for any such $\alpha<\kappa$.

    So fix $\alpha<\kappa$. Since $\mathrm{Col}(\omega,\kappa)$ is  homogeneous and  $\phi$ has no parameters, we have that for each $n \in \omega$, either 
    \[\emptyset \Vdash_{\mathrm{Col}(\omega,\kappa)} \forall y (\phi(\check{\alpha},y) \to \check{n} \in y),\] 
    or 
    \[\emptyset \Vdash_{\mathrm{Col}(\omega,\kappa)} \exists y (\phi(\check{\alpha},y) \wedge \check{n} \notin y).\] 
    Let 
    $$
    r = 
    \Mgd{ n \in \omega }
    { \emptyset \Vdash_{\mathrm{Col}(\omega,\kappa)} \forall y (\phi(\check{\alpha},y) \to \check{n} \in y)}
    $$
    Since ZFC proves that $\phi$ defines a function, it is immediate that 
    $$
    \emptyset\Vdash_{\mathrm{Col}(\omega,\kappa)} \phi(\check{\alpha},\check{r}).
    $$
    Therefore $\psi(\alpha,r)$ holds.  
\end{proof}

Observe that if $\phi(x,y)$ is a formula as in Lemma~\ref{lem:homogeneous-function}, then, as $\kappa>\mathfrak c$, the function defined by $\psi$ cannot be injective and therefore there are $\alpha<\beta<\kappa$ and $r\in 2^{\omega}$ so that
$$
\emptyset \Vdash_{\mathrm{Col}(\omega,\kappa)} \phi(\check{\alpha},\check{r}) \; \wedge\; \phi(\check{\beta},\check{r}),
$$
whereby
$$
V[G]\models  \phi(\alpha,r) \; \wedge\; \phi(\beta,r)
$$
for any choice of generic filter $G$.
Recall however that, in $V[G]$, both $\alpha$ and $\beta$ are countable ordinals.

The connection between ordinals and first-order theories is
captured by the following result of Mazurkiewicz and Sierpi\'nski. For this, if $d\geqslant 2$ and $\alpha$ is a countable ordinal, we let 
$$
M^d_{\alpha}=S^d\setminus F
$$
where $F$ is some closed subset of $S^d$ homeomorphic with the countable compact space $\omega^{\alpha}+1$. Note that, up to homeomorphism, $M^d_{\alpha}$ is independent of the specific choice of $F$.

\begin{thm}[Mazurkiewicz--Sierpi\'nski]\label{thm:cb-ms}
For all $d\geqslant 2$ and all countable ordinals $\alpha\neq \beta$, the spaces $M^d_{\alpha}$ and $M^d_{\beta}$ are not homeomorphic.
\end{thm}

Fix a canonical enumeration $\{\sigma_n\}_{n<\omega}$ of all $\mL$-sentences. We let $\vartheta(\alpha,n)$ be a formula expressing that $\alpha$ is a countable ordinal  such that 
$$
\Homeo(S^{(n)_1}\setminus F)\models\sigma_{(n)_2}
$$ 
for any $F\subseteq S^{(n)_1}$ that is homeomorphic to $\omega^{\alpha}+1$. Here $(n)_1$ and $(n)_2$ refer to the first and second coordinates of the image of $n$ under some recursive bijection $\omega\longrightarrow \omega\times \omega$.

Using $\vartheta$, we may define a function 
$\Phi\colon \omega_1\longrightarrow 2^\omega$ by setting
$$
\Phi(\alpha)=\mgd{n<\omega}{\vartheta(\alpha,n)}.
$$
Observe that $\Phi$ is defined by a formula $\phi(x,y)$ without parameters. Therefore, by our previous discussion, there are ordinals $\alpha\neq \beta$ that are countable in $V[G]$ and some $r\in 2^\omega$ such that
$$
V[G]\models  \phi(\alpha,r) \; \wedge\; \phi(\beta,r).
$$
It thus follows that in $V[G]$, the following equivalence holds for all $n$
$$
\Homeo(M^{(n)_1}_\alpha)\models\sigma_{(n)_2}  \;\;\Leftrightarrow \;\;\Homeo(M^{(n)_1}_\beta)\models\sigma_{(n)_2}
$$
and so 
$$
\Homeo(M^{d}_\alpha)\equiv\Homeo(M^{d}_\beta)
$$
for all $d\geqslant 2$.

Summing up, by forcing with $\mathrm{Col}(\omega,\mathfrak{c}^+)$, in the generic extension $V[G]$ there is a pair of non-homeomorphic orientable $d$-dimensional manifolds with elementarily equivalent homeomorphism groups. This establishes
Theorem~\ref{thm:consistent-intro}.


\section*{Acknowledgements}
TK is partially supported by NSF grant DMS-2349814, and was partially supported by
Simons Foundation International Grant SFI-MPS-SFM-00005890. JdlNG is partially supported
by KIAS Individual Grant MG084001 at Korea Institute for Advanced Study,
and by the Samsung Science and Technology Foundation under Project Number SSTF-BA1301-51. CR is partially supported by NSF grant DMS-2246986.

\bibliographystyle{amsplain}
  \bibliography{ref}

\end{document}